\documentclass[12pt]{article}

\usepackage{amsmath, latexsym, amsfonts, amssymb, amsthm}
\usepackage{graphicx, color,hyperref,dsfont,tikz,caption,wrapfig,subfig}
\usepackage{pgf,pgfarrows,pgfnodes,pgfautomata,pgfheaps,pgfshade,color}
\usepackage{movie15,hyperref}
\hypersetup{
    linktoc=page,
    linkcolor=red,          
    citecolor=blue,        
    filecolor=blue,      
    urlcolor=cyan,
    colorlinks=true           
}

\numberwithin{equation}{section}

\newtheorem{theorem}{Theorem}[section]
\newtheorem{lemma}[theorem]{Lemma}
\newtheorem{proposition}[theorem]{Proposition}
\newtheorem{corollary}[theorem]{Corollary}

\newtheorem{remark}[theorem]{Remark}

\theoremstyle{definition}

\DeclareMathSymbol{\leqslant}{\mathalpha}{AMSa}{"36} 
\DeclareMathSymbol{\geqslant}{\mathalpha}{AMSa}{"3E} 
\DeclareMathSymbol{\eset}{\mathalpha}{AMSb}{"3F}     
\renewcommand{\leq}{\;\leqslant\;}                   
\renewcommand{\geq}{\;\geqslant\;}                   

\newcommand{\R}{\mathbb{R}}
\newcommand{\Z}{\mathbb{Z}}
\newcommand{\N}{\mathbb{N}}
\newcommand{\Q}{\mathbb{Q}}
\newcommand{\E}{\mathds{E}}
\newcommand{\Pb}{\mathds{P}}
\newcommand{\ind}{\mathds{1}}

\newcommand{\0}{\boldsymbol{0}}
\newcommand{\1}{\boldsymbol{1}}
\newcommand{\bw}{\boldsymbol{w}}
\newcommand{\bz}{\boldsymbol{z}}
\newcommand{\bx}{\boldsymbol{x}}
\newcommand{\by}{\boldsymbol{y}}

\def\Bvec{\boldsymbol B}
\def\Wvec{\boldsymbol W}
\def\Gammavec{\boldsymbol \Gamma}
\def\Deltavec{\boldsymbol \Delta}
\def\Leb{\operatorname{Leb}}
\def\mubr{\mu_{\text{br}}}
\def\d{\mathbf{d}}
\def\eps{\varepsilon}
\def\T{\mathbb{S}}
\def\T{\mathbb{T}}

\def\bi{\begin{itemize}}
\def\ei{\end{itemize}}

\def\bnum{\begin{enumerate}}
\def\enum{\end{enumerate}}
\def\LB{\mathcal{B}} 
\def\<#1{\langle #1 \rangle}

\def\r{\mathbf{r}}

\def\p{\mathbf{p}}

\usepackage{xparse}

\DeclareDocumentCommand \Pmp { m m o} {
\IfNoValueTF{#3}
{P_{#1}^{#2}}
{P_{#1}^{#2}\left(#3\right)}
}

\DeclareDocumentCommand \Emp { m m o} {
\IfNoValueTF{#3}
{E_{#1}^{#2}}
{E_{#1}^{#2}\left[#3\right]}
}

\DeclareDocumentCommand \Pbr { m m m m o } {
\IfNoValueTF{#5}
{P_{#1}^{#2\stackrel{#4}{\rightarrow} #3}}
{P_{#1}^{#2\stackrel{#4}{\rightarrow} #3}\left(#5\right)}
}

\DeclareDocumentCommand \Ebr { m m m m o } {
\IfNoValueTF{#5}
{E_{#1}^{#2\stackrel{#4}{\rightarrow} #3}}
{E_{#1}^{#2\stackrel{#4}{\rightarrow} #3}\left[#5\right]}
}

\DeclareDocumentCommand \pr{ o } {
\IfNoValueTF{#1}
{\Xi}
{\Xi_{#1}}
}

\DeclareDocumentCommand \prx{ o } {
\IfNoValueTF{#1}
{\xi}
{\xi_{#1}}
}

\DeclareDocumentCommand \pry{ o } {
\IfNoValueTF{#1}
{\zeta}
{\zeta_{#1}}
}

\DeclareDocumentCommand \Loct { m m m } {L_{#1}^{#2}(#3)} 

\DeclareDocumentCommand \Ind { m } {\mathbf 1_{(#1)}}
\DeclareDocumentCommand \Indev { m } {\mathbf 1_{#1}}

\title{Liouville heat kernel: regularity and bounds}
\author{ Maillard P.\footnote{Weizmann Institute of Science, Rehovot, Israel. Partially supported  by a grant from the Israel Science Foundation}, 
Rhodes R.\footnote{Universit{\'e} Paris-Dauphine, Ceremade, F-75016 Paris, France. Partially supported by grant ANR-11-JCJC  CHAMU} \footnotetext[2]{Partially supported by grant ANR-11-JCJC  CHAMU},
Vargas V.\footnote{Ecole Normale Sup\'erieure, DMA, 45 rue d'Ulm,  75005 Paris, France. Partially supported by grant ANR-11-JCJC  CHAMU}, Zeitouni O.\footnote{Weizmann Institute of Science, Rehovot, Israel. Partially supported  by a grant from the Israel Science Foundation}}
 
\date{\empty}

\begin{document}

\maketitle
 
 
\begin{abstract}
We initiate in this paper the study of
analytic properties of
the Liouville heat kernel. In particular, we establish regularity estimates on the heat kernel and derive non trivial lower and upper bounds. 

\end{abstract}
\footnotesize



\noindent{\bf Key words or phrases:} Liouville quantum gravity, heat kernel, Liouville Brownian motion, Gaussian multiplicative chaos.

\noindent{\bf MSC 2000 subject classifications:    35K08, 60J60, 60K37, 60J55, 60J70}
\normalsize

\tableofcontents

  
\section{Introduction}

Liouville quantum gravity (LQG) in the conformal gauge was introduced by Polyakov in a 1981 seminal paper \cite{Pol} and can be considered as the canonical $2d$ random Riemannian surface. Indeed, physicists have long conjectured that LQG (which is parametrized by a constant $\gamma$) is the limit of random planar maps weighted by a $2d$ statistical physics system at critical temperature, usually described by a conformal field theory with central charge $c \leq 1$. These conjectures were made more explicit in a recent work \cite{cf:DuSh} and in particular they provide a geometrical and probabilistic framework for the celebrated KPZ relation (first derived in \cite{cf:KPZ} in the so-called light cone gauge and then in \cite{cf:Da, DistKa} within the framework of the conformal gauge): see \cite{Rnew4,benj, Rnew12,cf:DuSh, Rnew10} for rigorous probabilistic formulations of the KPZ relation. In this geometrical point of view, (critical) LQG corresponds to studying a conformal field theory with central charge $c \leq 1$ (called the matter field in the physics literature) in an independent random geometry which can be described formally by a Riemannian metric tensor of the form   
\begin{equation}\label{tensor}
e^{\gamma X(x)}\,dx^2
\end{equation}
where $X$ is a Gaussian Free Field (GFF) on (say) the torus $\T$ and $\gamma$ a parameter in $[0,2]$ related to the central charge by the celebrated KPZ relation \cite{cf:KPZ} 
\begin{equation*}
\gamma=\frac{\sqrt{25-c}-\sqrt{1-c}}{\sqrt{6}}.
\end{equation*}
Of course, the above formula \eqref{tensor} is non rigorous as the GFF is not a random function hence making sense of \eqref{tensor} is still an open question. In particular, LQG can not strictly speaking be endowed with a classical  Riemannian metric structure as the underlying geometry is too rough and requires regularization procedures to be defined.    

Nonetheless, one can make sense of the volume form $M_\gamma$ 
associated to \eqref{tensor} by the theory of Gaussian multiplicative 
chaos \cite{cf:Kah}. Recently, the authors of \cite{GRV1} introduced the 
natural diffusion process $(\LB_{t})_{t \geq 0}$ associated to 
\eqref{tensor}, the so-called Liouville Brownian motion (LBM) (see also the 
work \cite{berest} for a construction of 
the LBM starting from one point) but also in \cite{GRV2} the associated heat kernel 
$\p^\gamma_t(x,y)$ 
(with respect to $M_\gamma$), called 
the Liouville heat kernel. 
One of the main motivations behind the introduction of the LBM and 
more specifically the Liouville heat kernel is to get 
an insight into the geometry of LQG: indeed, one can for 
instance note that there is a sizable physics literature 
in this direction (see the book \cite{amb1} for a review). 
The purpose of this paper is thus to initiate a thorough study of 
the Liouville heat kernel. More precisely, we will
show that the heat kernel $\p^\gamma_t(x,y)$ is continuous 
as a function of the variables $(t,x,y) \in \R^* \times \T^2$, and then
obtain estimates on the heat kernel; this can be seen as
a first step in a more ambitious program devoted to the
derivation of 
precise estimates on $\p^\gamma_t(x,y)$. Note that the continuity we prove, the fact that the support of $M_\gamma$ is full and the strict positivity of $\p^\gamma_t(x,y)$ allows one to define the Liouville Brownian bridge between any fixed $x$ and $y$.

We recall that, on a standard smooth Riemannian manifold, 
Gaussian heat kernels estimates  in terms of the associated 
Riemannian distance  have been established, see
\cite{griRie} for a review.
Thereafter, heat kernel estimates have been obtained 
in the more exotic 
context of diffusions on (scale invariant) fractals: see \cite{barlow,HamKug} 
for instance.
In the context of LQG, it is natural to wonder what is the 
shape of the heat kernel and it is  difficult to draw a clear 
expected picture: first  because of the multifractality of 
the geometry and second because the existence of the distance  
$\d_\gamma$ associated to \eqref{tensor} remains one of  the main 
open questions in LQG (though there has been some progress in the case 
of pure gravity, i.e. $\gamma=\sqrt{\frac{8}{3}}$: see \cite{MS} where the 
authors construct 
the analog of growing quantum balls without proving the existence of the 
distance $\d_\gamma$).

\subsubsection*{Brief description  of  the results} 
Our main lower bound on the heat kernel reads as follows:
if $x,y$ and $ \eta>0$ are fixed, one can find a random
time $T_0>0$ (depending on the GFF $X$ and $x,y,\eta$) such that, 
for any $t\in ]0,T_0]$,
\begin{equation*}
 \p^\gamma_t(x,y)  \geq  \exp \big (-t^{-\frac{1}{1+\gamma^2/4-\eta}} \big).
\end{equation*}  
This is the content of Theorem \ref{th:lower_bound} below. 
We emphasize that the exponent $1/(1+\gamma^2/4)$ is not 
expected to be optimal,
as in our derivation
we do not take into account the geometry of the Gaussian field.

For $\gamma^2\leq 8/3$, the same heat kernel lower bound 
holds when the endpoints
are sampled according to the measure $M_\gamma$. This is proven in Section~\ref{subsampling1} for $\gamma^2 \le 4/3$ and extended to  $\gamma^2 \le 8/3$ in Section~\ref{subsampling2}. For $\gamma^2>8/3$, a worse lower bound still holds and is proven in Section~\ref{subsampling2} as well.

We will also give the following
uniform upper bound on the heat kernel 
(see Theorem \ref{upperG} below for a precise statement): for all $\delta>0$ there exists $\beta=
\beta_\delta(\gamma)$ and some random constants $c_1,c_2>0$ (depending on the GFF $X$ only) such  that
\begin{equation*}
\forall x,y\in\T,t>0,\quad \p^\gamma_t(x,y)\leq \left(\frac{c_1}{t^{1+\delta}}
+1\right)\exp\Big(-c_2 \left ( \frac{d_{\T}(x,y)}{t^{1/ \beta}}  \right )^{\frac{\beta}{\beta-1}} \Big).
\end{equation*}   
where $d_{\T}$ is the standard distance on the torus.
There is a gap between our lower and upper bound for the (inverse) power 
coefficient of $t$ in the exponential: see Figure \ref{figdim} for a plot as a function of $\gamma$ (the graph of the upper bound corresponds to 
the limit of $\beta_\delta$ as $\delta$ goes to 0). \footnote{After this work was completed and posted, \cite{andres} obtained an improvement of the upper 
bound presented in this work, in both the  on and off diagonal regimes.}

Note that our estimates on the heat kernel, and in particular the upper bound,
are given
in terms of the Euclidean distance. The lower/upper bounds that we obtain
do not match but this does not come as a surprise because such a matching   
would mean in a way that  $\d_\gamma\asymp (d_{\T})^\theta$ for some exponent $\theta>0$, which is not expected.  
Yet our results illustrate that we can read off the Liouville heat kernel  
some uniform H\"older control of the geometry of LQG in terms of the 
Euclidean geometry. This was already known for the Liouville measure 
by means of multifractal analysis (see \cite{review} for a precise 
statement and further references). To our knowledge, our work is one of the 
first to investigate the problem of heat kernel estimates 
in a multifractal context, in sharp contrast with the monofractal framework 
of diffusions on fractals.   Notice however that on-diagonal heat kernel estimates have also been investigated in the context of one-dimensional multifractal geometry, see \cite{barkum}.

We conclude with some cautionary remarks on the (non)-sharpness of our methods.
In both the lower and upper bound, we have not taken much advantage of the
geometry determined by the GFF. In particular, our upper bounds are 
\textit{uniform} on the torus, and thus certainly not tight for typical
points. Similarly, in the derivation of our lower bound, we essentially force 
the LBM to follow a straight line between the starting and ending points. 
It is natural to expect that forcing the LBM to follow a path adapted to the 
geometry of the GFF could yield a better lower bound.

\subsubsection*{Discussion and speculations}
Here we develop a short speculative discussion that  
has motivated at least partly our study. It has been suggested by Watabiki \cite{wata,amb} that the (conjectural) metric space  $(\T,\d_\gamma)$ is locally monofractal with intrinsic Hausdorff dimension
\begin{equation}\label{watabikini}
d_H(\gamma)=1+\frac{\gamma^2}{4}+\sqrt{\big(1+\frac{\gamma^2}{4}\big)^2+
\gamma^2},
\end{equation}
(note that in the special case of pure gravity $\gamma=\sqrt{\frac{8}{3}}$ 
this gives $d_H(\gamma)=4$ which is compatible with the dimension 
of the Brownian map).

By analogy with the literature on fractals, it is natural to 
conjecture\footnote{Our results do not shed light on this conjecture, 
	nor on whether the 
  various parameters in \eqref{scalingconjecture} might be different for
points $x,y$ sampled according to $M_\gamma$.}
the following   asymptotic  expression ($t\to 0$) 
\begin{equation}\label{scalingconjecture}
\p^\gamma_t(x,y)\asymp \frac{C}{t^{\frac{d_H(\gamma)}{\beta}}}\exp\Big(-c\frac{\d_\gamma(x,y)^{\frac{\beta}{\beta-1}}}{t^{\frac{1}{\beta-1}}}\Big).
\end{equation}        
where $C,c>0$ are some global constants (possibly random), 
$\beta>0$ some exponent and $\asymp$ means that $\p_t^\gamma$ is 
bounded from above and below by two such expressions with possibly 
different values of $c,C$. Relation \eqref{scalingconjecture} should 
be understood for $t$ less than some random threshold $T$ 
(depending on the free field $X$, and possibly also on $x,y$). 
The ratio $\frac{2 d_H(\gamma)}{\beta}$, 
called the spectral dimension of LQG, is equal to $2$: 
this has been heuristically computed   by Ambj\o rn and al. in 
\cite{amb}  and then rigorously derived in a weaker form in \cite{spectral}. 
Assuming 
\eqref{scalingconjecture},
this would yield the relation $$d_H(\gamma)=\beta.$$ 
Further, still assuming 
\eqref{scalingconjecture},
one would obtain, for any $x\neq y$,
\begin{equation}
	\label{scalerel-1}
	d_H=1+	\limsup_{t\to 0}\frac{\log t}{\log \log \p^\gamma_t(x,y)}
	=
1+	\liminf_{t\to 0}\frac{\log t}{\log \log \p^\gamma_t(x,y)}
	\,.
\end{equation}

The results in this article give precise upper and lower
bounds on the expressions in the right side of 
	\eqref{scalerel-1}, which could be interpreted as bounds
	on $d_H$ if one accepts the ansatz 
\eqref{scalingconjecture}. Those bounds are plotted in
Figure~\ref{figdim}, together with the Watabiki conjecture
\eqref{watabikini}.
Note that
Watabiki's formula for 
 $d_H(\gamma)$ lies somewhere between our lower 
 and upper bound.
\begin{figure}[ht]
\begin{center}
\includegraphics[height=10cm]{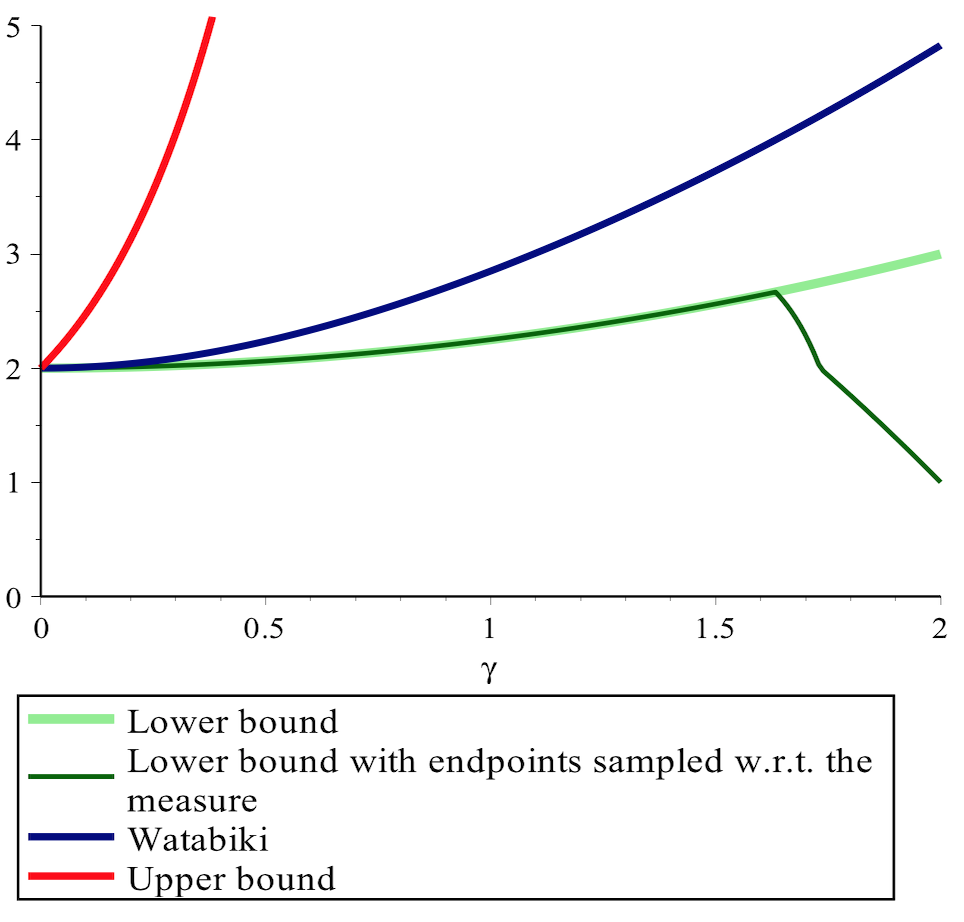}
\caption{Bounds on $d_H(\gamma)$ assuming \eqref{scalingconjecture}. 
Note that our bounds do not shed light on whether \eqref{scalingconjecture} is true.}
\label{figdim}
\end{center}
\end{figure}

%
%
%

 \subsubsection*{Organization of the paper} 
In the next section, we introduce our setup: for technical reasons, we work on the torus $\T$ though most of our 
results extend to other setups like the plane or the sphere.
In Section \ref{studyHK}, we construct a representation for the Liouville 
heat kernel using a classical Hilbert-Schmidt decomposition, and obtain
regularity estimates for the heat kernel. In Sections \ref{sec-UB} and \ref{sec-LB}, we give (uniform)  upper and lower bounds for the kernel.

\subsubsection*{Acknowledgements}
The authors thank M. Barlow,  T. Budd, F. David, J. Ding and M. Takeda for useful discussions.
\section{Setup}

\subsection{Notation}

We equip  the two dimensional torus  $\T$ with its standard Riemann distance   
$d_\T$ and volume form $dx$ (also called Lebesgue measure on $\T$). 
We denote by $B(x,r)$ the ball centered at $x$ with radius $r$. 
The standard spaces $L^p(\T,dx)$ are denoted by $L^p$. 
$C(\T)$ stands for the space of continuous functions on $\T$.  

Denote by
$\triangle$  
the Laplace-Beltrami operator on $\T$ and by  $p_t(x,y)$ 
the standard heat kernel  of the  Brownian motion ${\Bvec}$ on $\T$. 
Recall that $p_t(x,y)$  can be written in the following form
\begin{equation*}
p_t(x,y)= \frac{1}{|\T|}+ \sum_{n \geq 1}   e^{-\lambda_n t} e_n(x) e_n(y)
\end{equation*}
where $(\lambda_n)_{n \geq 1}$ and $(e_n)_{n \geq 1}$  
are respectively the eigenvalues and eigenvectors  
of (minus) the standard Laplacian:
\begin{equation*}
- \Delta e_n =\lambda_n e_n, \quad \quad \int_{\T} e_n(x) dx= 0 .
\end{equation*}
We use the convention that the $(\lambda_n)_{n \geq 1}$ are increasing;
by Weyl's formula,
$\lambda_n \underset{n \to \infty}{\sim} n$. We denote by 
\begin{equation}\label{green}
G(x,y)=\sum_{n\geq 1}\frac{1}{\lambda_n}e_n(x)e_n(y)
\end{equation}
the standard Green function of the Laplacian $\Delta$ on 
$\T$ with vanishing mean.

Throughout the article, the symbols $C,C',C''$ etc.\ stand for positive constants whose value may change from line to line and which may be random when mentioned.
 
\subsection{Log-correlated Gaussian fields}    
Throughout this paper,
$X$ stands for any centered log-correlated Gaussian field 
(LCGF for short) of $\sigma$-positive type \cite{cf:Kah} on the torus. 
Its covariance kernel takes the form
\begin{equation}\label{covX}
\E^X[X(x)X(y)]=\ln_+\frac{1}{d_\T(x,y)}+g(x,y)
\end{equation}
where $\ln_+(u)= \max(0,\ln u)$ for $u\in \R_+^*$ 
and $g$ is a continuous bounded function on $\T^2$. 
We denote by $\Pb^X$ and $\E^X$ the law and expectation 
with respect to the LCGF $X$.

A particularly important example is that of
the Gaussian Free Field (GFF for short)   with vanishing average, 
that is the  centered Gaussian random distribution with covariance 
$2\pi G$, where $G$ is the Green function 
given by \eqref{green} \cite{dubedat,glimm,She07}.  

\subsection{Liouville measure and Liouville Brownian motion}\label{LMLBM}
We fix $\gamma \in [0,2[$ and consider the Gaussian multiplicative chaos 
\cite{cf:Kah,review} with respect to the Lebesgue measure $dx$, which is formally defined by
\begin{equation}
M_\gamma(dx)=e^{ \gamma X(x)-\frac{\gamma^2}{2}\E[X(x)^2]}\,dx.
\end{equation}
Recall,
see e.g. \cite[Theorem 2.14]{review}, 
that for some universal deterministic constant $C$,
we have, for  
$p< \frac{4}{\gamma^2}$ and all $x\in \T$, that
\begin{equation}\label{powerlaw}
\E^X[  M_\gamma(B(x,r))^p ]  \underset{r \to 0}{\sim}^C   r^{\zeta(p)}
\end{equation}
where 
$\zeta(p)=(2+\frac{\gamma^2}{2})p -\frac{\gamma^2}{2}p^2$ and  
$a\underset{r\to 0}\sim^C b$ 
means that $\limsup_{r\to 0} (a/b\vee b/a)\leq C$.
$\zeta(p)$ is referred to as
the \textit{power law spectrum} of the measure $M_\gamma$. The following Chernoff inequality,
\begin{equation}
\label{eq:chernoff}
\Pb^X(M_\gamma(B(x,r)) \geq r^{2+\frac{\gamma^2}{2}-\gamma a}) \leq C_a r^{\frac{a^2}{2}},\quad\forall a\ge 0,
\end{equation}
is then readily obtained from \eqref{powerlaw} by setting $p=a/\gamma$ and using Markov's inequality.
 
Further, it is proved in \cite{GRV1} that the measure
$M_\gamma$ is  H\"older continuous: 
for each $\epsilon>0$,   $\Pb^X$-a.s., there is a random constant $C=C(\epsilon,X)$  such that
\begin{equation}\label{holderM}
\forall x\in\T,\forall r>0,\quad M_\gamma(B(x,r))\leq Cr^{\alpha-\epsilon},
\end{equation}
with $\alpha=2\big(1-\frac{\gamma}{2}\big)^2$. 

We denote by $L^2_\gamma$ the Hilbert space $L^2(\T,M_\gamma(dx))$.  
We also use the standard notation $L^p_\gamma$ for the spaces 
$L^p(\T,M_\gamma(dx))$ for $1\leq p \leq \infty$.   
We denote by $L^p_{\gamma,0}$ the closed subspace of $L^p_\gamma $ 
consisting of functions $f$ such that $\int_{\T}f(x)M_\gamma(dx)=0$. 
As $M_\gamma$ is a Radon measure on the Polish space $\T$, 
the spaces $L^p_\gamma$ are separable for $1\leq p < \infty$, 
with $C(\T)$ as dense subspace.

We also consider the associated  
Liouville Brownian Motion (LBM for short, see \cite{GRV1}).
More precisely, we consider in the same probability space the
LCGF $X$ and a  Brownian motion $\Bvec = (\Bvec_t)_{t\geq0}$ on $\T$, independent
of the LCGF.

We denote by $\Pmp{\Bvec}{x}$ and $\Emp{\Bvec}{x}$ the probability 
law and expectation of this Brownian motion when starting from $x$. 
With a slight abuse of notation, we also denote by  
$\Pbr{\Bvec}{x}{y}{t}$ and $\Ebr{\Bvec}{x}{y}{t}$ the law and 
expectation of the Brownian bridge $({\Bvec}_s)_{0\leq s \leq t}$ 
from $x$ to $y$ with lifetime $t$. We will apply in the sequel
the same convention 
to possibly other stochastic processes $\Bvec$. 
We  also introduce the annealed probability laws 
$\Pb_x=\Pb^X\otimes \Pmp{\Bvec}{x}$ and the corresponding expectation $\E_x$. 

We also consider ($\Pb^X$-almost surely) the unique Positive
Continuous Additive Functional (PCAF) $F$ associated to the 
Revuz measure $M_\gamma$, which is defined under $\Pmp{\Bvec}{x}$ 
for all starting point $x\in\T$ (see \cite{fuku} for the terminology and \cite{GRV1} for further details in our context).
Then, $\Pb^X$-almost surely, the law of the LBM  under $\Pmp{\Bvec}{x}$  
is given by
$$ \LB_t={\Bvec}_{F(t)^{-1}} $$
for all $x\in\T$. Furthermore, this PCAF can be understood 
as a Gaussian multiplicative chaos with respect to the 
occupation measure of the Brownian motion ${\Bvec}$ 
\begin{equation}\label{def:F}
F(t)=\int_0^t e^{ \gamma X({\Bvec}_r)-\frac{\gamma^2}{2}\E^X[X^2({\Bvec}_r)]}\,dr.
\end{equation}
The following facts concerning the LBM are  detailed in \cite{GRV1,GRV2}.
The LBM is a  Feller Markov process with continuous sample paths    
and associated semigroup
$(P^\gamma_t)_{t\geq 0}$ and resolvent $(R_\lambda)_{\lambda > 0}$ which we refer to as the
Liouville 
semigroup and resolvent, respectively. Further, 
$\Pb^X$-almost surely, this semigroup is absolutely continuous  
with respect to the Liouville measure $M_\gamma$  and 
there exists a measurable function $ \p^\gamma_t(x,y)$, referred to as the
Liouville heat kernel,
such that for all $x\in\T$ and any measurable bounded function $f$ 
\begin{equation}
P^\gamma_tf(x)=\int_{\T}f(y)\p^\gamma_t(x,y)\,M_\gamma(dy).
\end{equation}

The following bridge formula, established in \cite{spectral}, will be useful  in our analysis  of
the lower bound.
\begin{theorem}\label{generalformula}
$\Pb^X$-almost surely, for each $x,y\in\T$ and any continuous function 
$g:\R_+\to\R_+$
\begin{equation} \label{transformBB}
\int_{0}^\infty  g(t)    \p^\gamma_t(x,y) dt =   \int_0^\infty  
\Ebr{\Bvec}{x}{y}{t}[ g(  F(t))  ]  p_t(x,y) dt  .
\end{equation}
\end{theorem}
With the choice 
$g(t)=e^{- \lambda t}$ and $\lambda>0$, we thus obtain a representation of the
Liouville resolvent
\begin{equation}
\label{eq:resolvent}
\r^\gamma_\lambda(x,y):=\int_0^\infty e^{-\lambda t}
   \p^\gamma_t(x,y) dt =   
\int_0^\infty \Ebr{\Bvec}{x}{y}{t}[e^{-\lambda F(t)}]p_t(x,y)\,dt.
\end{equation}
We define for a Borel set $A\subset\T$,
\[
\r^\gamma_\lambda(x,A)=\int_A \r^\gamma_\lambda(x,y)\,M_\gamma(dy) = R^\gamma_\lambda \Indev{A}(x).
\]

It is proved in \cite{spectral}   that
$\r^\gamma_\lambda(x,y)$ is a 
continuous function of $\lambda$ and $x\not = y$.
It is also proved there that for any
$\delta>0$, 
\begin{equation}
        \label{eq-4.3a}
        \mbox{\rm the function
$(x,y)\mapsto \int_0^1 t^\delta \p_t^\gamma(x,y)\,dt$ 
is  continuous on $\T^2$}.
\end{equation}
In particular,
\begin{equation}
        \label{eq-4.3}
\sup_{x\in\T}\int_0^1 t^\delta \p_t^\gamma(x,x)\,dt<+\infty.
\end{equation}

\begin{remark}
To be precise,  \eqref{eq-4.3a} is established in \cite{spectral} for the LBM on the whole plane. We give here a short explanation how to adapt the argument to the torus. First, for $\delta \in]0,1]$ and $x\not =y$, we apply Theorem \ref{generalformula} to get
$$\int_0^\infty t^\delta e^{-\alpha t}\p_t^\gamma(x,y)\,dt=\int_0^{\infty}\Ebr{\Bvec}{x}{y}{t}[ F(t)^\delta e^{-\alpha F(t)}]p_t(x,y)\,dt.$$
Then we split this latter integral into two parts
\begin{align*}
\int_0^\infty t^\delta e^{-\alpha t}\p_t^\gamma(x,y)\,dt=&\int_0^{1}\Ebr{\Bvec}{x}{y}{t}[ F(t)^\delta e^{-\alpha F(t)}]p_t(x,y)\,dt       \\&+\int_1^{\infty}\Ebr{\Bvec}{x}{y}{t}[e^{-\alpha F(t)}F(t)^\delta]p_t(x,y)\,dt.
\end{align*}
The main difference between the torus and the whole plane is the long-time behaviour of the standard heat kernel $p_t(x,y)$. Therefore the first integral can be treated as in \cite[subsection 3.2, eq. (3.9)]{spectral}. Concerning the second integral, we use the following  facts: 

1)  $ F(t)^\delta e^{-\alpha F(t)}\leq C e^{-\alpha F(t/2)/2}$  , 

2) for all $x,y\in\T$ and $t\geq 1$,  $p_t(x,y)\leq C$  

3) the absolute continuity of the Brownian bridge (see \cite[Lemma 3.1]{spectral}).

By the strong Markov property of the Brownian motion, we get 
$$  \sup_{z\in\T} \Emp{\Bvec}{z}[e^{-\alpha F(t)}] \,dt\leq \big(\sup_{  z \in\T}  \Emp{\Bvec}{z}[ e^{-\alpha  F(1)}]  \big)^{  \lfloor   t  \rfloor  }.$$
Because the mapping    $z\mapsto \Emp{\Bvec}{z}[e^{-\alpha  F(1)}]$  is continuous \cite{GRV1}, the supremum is reached at some point $z_0$   and because we have $F(1)>0$     $\Pmp{\Bvec}{z_0}$-almost   surely, we deduce  $\sup_{z\in\T}\Emp{\Bvec}{z}[e^{-\alpha  F(1)}]<1$.
This concludes the argument.
\end{remark}

\section{Representation and regularity of the heat kernel on the torus}    \label{studyHK}
In this section we derive
regularity properties of the Liouville heat kernel.  

We begin by establishing
a spectral representation of the Liouville heat kernel 
following a somewhat standard procedure: the reader 
may consult \cite{buser} for the case of compact Riemannian manifolds 
or \cite{kigami} for the case of heat kernels on fractals. 
 This will be useful  in obtaining further   properties of the heat kernel. 
It is proved in \cite{GRV2} that the Green function of the 
LBM coincides with $G$ of \eqref{green} up to  recentering  the mean, namely 
\begin{equation}\label{greenop}
M_\gamma(dx)\text{-a.s.},\quad \int_0^\infty P^\gamma_tf(x)\,dt=\int_\T G_\gamma(x,y) f(y) M_\gamma(dy)
\end{equation}
with
\begin{equation}\label{greenop2}
G_\gamma(x,y) =G(x,y)-\frac{\int_\T G(z,y)M_\gamma(dz)}{ M_\gamma(\T)}
\end{equation}
for every function $f\in L^1_{\gamma,0}$. 
 We now have the following (write $x^p = \operatorname{sgn}(x)|x|^p$ for $x\in\R$):
\begin{lemma}\label{holder}
Assume that $\mu$ is a measure on $\T$ such that for some $\theta>0$, we have $\mu(B(x,r))\leq 
C r^\theta$ for all $r\le 1$ and $x\in\T$. Then, for all $p>0$ and any 
bounded measurable function $f$ on $\T$, the mapping $x\mapsto 
\int_\T G(x,y)^pf(y)\mu(dy)$ is a continuous function of $x$ and hence
bounded 
on $\T$.
\end{lemma}

\noindent {\it Proof.} Choose a continuous function 
$\varphi:\R_+\to\R_+$ such that 
$0\leq \varphi\leq 1$, $ \varphi(u)=0$ if 
$|u|\leq 1$ and $ \varphi(u)=1$ if $|u|\geq 2$ .  Observe that, for any
$\delta>0$,
\begin{align*}
\int_\T &G(x,y)^pf(y)\mu(dy)\\
=&\int_\T G(x,y)^p\varphi (|x-y|/\delta)f(y)\mu(dy)+\int_\T G(x,y)^p\big(1-\varphi (|x-y|/\delta)\big)f(y)\mu(dy)\\
=:& \,\,A_\delta(x)+B_\delta(x).
\end{align*}
For each $\delta>0$, the mapping 
$x\mapsto A_\delta(x)$ 
is continuous, since $G(\cdot,\cdot)$ is continuous off-diagonal. 
Therefore, it suffices to prove that
\begin{equation}
\lim_{\delta\to 0}\sup_{x\in\T}|B_\delta(x)|=0.
\end{equation}
To see this, we write
\begin{align*}
|B_\delta(x)|&\leq \|f\|_\infty
\sup_{\T}\int_{|x-y|\leq 2\delta } G(x,y)^p \mu(dy)\\
&\leq  \|f\|_\infty \sup_{\T}\sum_{n\geq 1} 
\int_{2^{-n-1}\delta \leq |x-y|\leq 2^{-n}\delta } G(x,y)^p \mu(dy)\\
&\leq  C \|f\|_\infty
\sup_{\T}\sum_{n\geq 1} 
\big(\ln (2^{n+1}/\delta)\big)^p \mu(B(x,2^{-n}\delta))\\
&\leq C \|f\|_\infty \sup_{\T}\sum_{n\geq 1} 
\big((n+1)\ln  2+|\ln \delta|)^p \delta^\theta 2^{-n\theta}\\
&\leq  2^p C \|f\|_\infty \sup_{\T}\delta^\theta(1+|\ln\delta|^p)
\sum_{n\geq 1} \big((n+1)^p(\ln  2)^p+1)  2^{-n\theta}.
\end{align*}
This latter quantity is independent of $x$ and clearly converges to $0$ as 
$\delta\to 0$.\qed

Lemma \ref{holder} implies that 
\begin{equation} 
\int_{\T}\int_{\T}G(x,y)^2M_\gamma(dx)M_\gamma(dy)<+\infty\quad \text{and }\quad \sup_{y\in\T}\Big|\int_{\T}G(z,y) M_\gamma(dz)\Big|<+\infty
\end{equation}
so that
\begin{equation}\label{HS1}
\int_{\T}\int_{\T}G_\gamma(x,y)^2M_\gamma(dx)M_\gamma(dy)<+\infty,
\end{equation}
and therefore the
operator 
\begin{equation}
	\label{eq-3.5}
T_\gamma: f\in L^2_{\gamma,0}\mapsto T_\gamma f(x)=
\int_{\T}G_\gamma(x,y)f(y)M_\gamma(dy)\in L^2_{\gamma,0}.
\end{equation}
is Hilbert-Schmidt. Indeed, this operator is Hilbert-Schmidt on $L^2_\gamma$ because of \eqref{HS1} so that its restriction to the stable subspace $L^2_{\gamma,0}$ is (note that
$T_\gamma$ does map $L^2_{\gamma,0}$ into $L^2_{\gamma,0}$: this can be seen 
thanks to \eqref{greenop} and the invariance 
of $M_\gamma$ for the semigroup $(P^\gamma_t)_t$ or just by computing the mean of  $T_\gamma f$ with the help of \eqref{greenop}+\eqref{greenop2}). We stress that $T_\gamma$ is self-adjoint on $L^2_{\gamma,0}$ (though it is not on $L^2_{\gamma}$).

\begin{lemma}\label{kernel}
The kernel of $T_\gamma$ on $L^2_{\gamma,0}$ 
consists of the  null function only.
\end{lemma}

\noindent {\it Proof.} Let us consider $f\in L^2_{\gamma,0}$ such that $T_\gamma f=0$. Then
\begin{align*}
0&=T_\gamma f(x)=\int_{\R^2}G_\gamma(x,y)f(y)M_\gamma(dy)  =\int_0^\infty P^\gamma_tf(x)\,dt.
\end{align*}
Integrating against $f(x)M_\gamma(dx)$ and using the symmetry of the semigroup of the LBM, we get
$$0=\int_0^\infty \Big(\int_\T f(x)P^\gamma_tf(x)M_\gamma(dx)\Big)\,dt=\int_0^\infty \Big(\int_\T |P^\gamma_{t/2}f(x)|^2M_\gamma(dx)\Big)\,dt.$$
Therefore, for Lebesgue almost every $t\geq 0$, we have $P^\gamma_tf=0$. Since the semigroup is strongly continuous, we deduce that $f=0$ $M_\gamma(dx)$-almost surely. \qed

\medskip
Since $T_\gamma$ is Hilbert-Schmidt and symmetric on the separable space 
$L^2_{\gamma,0}$, there exists an orthonormal basis 
$({\bf e}^\gamma_n)_{n\geq 1}$ of $L^2_{\gamma,0}$ made up of 
eigenfunctions of $T_\gamma f$. From Lemma \ref{kernel}, the 
associated eigenvalues are non null and we can consider the 
sequence $({\bf \lambda}_{\gamma,n})_n$ made up of inverse 
eigenvalues (i.e. ${\bf \lambda}_{\gamma,n}^{-1}$ is an eigenvalue) 
associated to $({\bf e}^\gamma_n)_{n\geq 1}$  in increasing order 
(${\bf \lambda}_{\gamma,1}\leq {\bf \lambda}_{\gamma,2}\leq\dots$).  
Because $T_\gamma$ is Hilbert-Schmidt, we have that
$\sum_n{\bf \lambda}_{\gamma,n}^{-2}<+\infty$; 
we will see below, see \eqref{eq-april28}, that a better 
estimate is available.

\begin{theorem}\label{heat}
The heat kernel $ \p ^\gamma$  associated to 
the LBM on $\T$  admits the representation
\begin{equation}
        \label{eq-insidethm3.3}
        \p^\gamma_t(x,y)=\frac{1}{M_\gamma(\T)} +
\sum_{n\geq 1}e^{-\lambda_{\gamma,n} t}{\bf e}^\gamma_n(x)
{\bf e}^\gamma_n(y). 
\end{equation}
Furthermore, it is of class $C^{\infty,0,0}(\R_+^*\times \T^2)$. 
If $\gamma<2-\sqrt{2}$, it is even of class 
$C^{\infty,1,1}(\R_+^*\times \T^2)$.
\end{theorem}

\noindent {\it Proof.} We know by Theorem 6.2.1 in \cite{fuku} 
that the Liouville semigroup $(P^\gamma_t)_{t\geq 0}$ is a strongly 
continuous semigroup of self-adjoint contractions on $L^2_\gamma$, which
furthermore preserves $L^2_{\gamma,0}$ due to \eqref{eq-3.5}.
Furthermore, all the operators $(P^\gamma_t)_{t\geq 0}$ 
commute with $T_\gamma$ and because all the eigenspaces of $T_\gamma$ 
are finite dimensional, we may assume without loss of generality 
that the family $({\bf e}^\gamma_n)_{n\geq 1}$ is a family 
of eigenfunctions of the operators $(P^\gamma_t)_{t\geq 0}$ too.

Then, for each $n\geq 1$, we can find a    
continuous function $a_{n}:\R_+\to\R_+$ such that
$$P^\gamma_t({\bf e}^\gamma_n)= a_{n}(t){\bf e}^\gamma_n.$$
From the semigroup property, we have $a_n(t)=e^{-c_n t}$ for some 
$c_n\geq 0$. 
Further, by using the relation
$$\int_0^\infty P^\gamma_t({\bf e}^\gamma_n)\,dt =
T_\gamma({\bf e}^\gamma_n)=\lambda_{\gamma,n}^{-1} {\bf e}^\gamma_n$$ 
we deduce that $P^\gamma_t({\bf e}^\gamma_n)=
e^{-{\bf \lambda}_{\gamma,n}t}{\bf e}^\gamma_n$. This implies the
representation \eqref{eq-insidethm3.3}. 

Using \eqref{eq-4.3} one then obtains
\begin{equation}
        \label{eq-newsunday2}
        \infty > \int_0^1 t^\delta \p_t^\gamma(x,x)dt \geq
\sum_n \frac{|{\bf e}_n^\gamma(x)|^2}{\lambda_{\gamma,n}^{1+\delta}} \times \int_0^{\lambda_{\gamma,1}} t^\delta e^{-t}\,dt\,.
\end{equation}
In particular, for any $\delta>0$ there exists a random constant 
$C_\delta$ so that 
\begin{equation}
        \label{CS}
        |{\bf e}_n^\gamma(x)|\leq C_\delta \lambda_{\gamma,n}^{(1+\delta)/2}.
\end{equation}
Also, integrating \eqref{eq-newsunday2}
with respect to $M_\gamma(dx)$ one gets  that for any
$\delta>0$,
\begin{equation}
\label{eq-april28}
\sum \lambda_{\gamma,n}^{-(1+\delta)}<\infty,
\end{equation}
which, together with the fact that the sequence $\lambda_{\gamma_n}$ is increasing, yields for any $\delta > 0$,
\begin{equation}
 \label{eq:miss-kittin}
 \lambda_{\gamma,n} \geq C_\delta' n^{1-\delta},
\end{equation}
for some random constant 
$C_\delta'$.

Now we show that the eigenfunctions $({\bf e}^\gamma_n)_n$ are continuous. 
The proof is based on the relation
\begin{equation}\label{eigen}
{\bf e}^\gamma_n(x)=\lambda_{\gamma,n}   
\int_{\T}G_\gamma(x,y){\bf e}^\gamma_n(y)M_\gamma(dy)\,,\; n\geq 1.
\end{equation}
By \eqref{CS}, \eqref{holderM} and Lemma \ref{holder}, one 
deduces that the mapping 
$x\mapsto   \int_{\T}G_\gamma(x,y){\bf e}^\gamma_n(y)M_\gamma(dy)$ 
is continuous on $\T$, and therefore, by
\eqref{eigen}, one concludes that
the eigenfunction ${\bf e}^\gamma_n(x)$ is a 
continuous function of $x$. 
Notice that for $\gamma<2-\sqrt{2}$, the exponent 
$\alpha$ in \eqref{holderM} satisfies $\alpha>1$ so that we 
can even integrate a $|x|^{-1}$-singularity instead of a 
log-singularity, leading to $C^1$-regularity of the eigenfunctions 
in that regime.

Finally we prove the continuity of the heat kernel. It suffices to 
establish the uniform convergence of the series; the latter
however follows immediately
from \eqref{eq-newsunday2}. Note
that this argument shows that the heat kernel is $C^\infty$ 
with respect to $t\in(0,\infty)$ 
with time derivatives that are continuous functions of $(t,x,y)
\in (0,\infty)\times \T^2$.\qed

\begin{corollary}
For each fixed $t_0>0$ there exists a random constant $C=C(X,t_0)$ so that
for all $t\geq s\geq t_0$ and $(x,y,x',y')\in \T^4$,
$$ |\p^\gamma_t(x,y)-\p^\gamma_s(x',y')|\leq C\Big(|t-s|+h(x,x')+h(y,y')\Big)$$
where $$h(x,x')=\int_\T|G(x,z)-G(x',z)|\,M_\gamma(dz).$$
\end{corollary}

\noindent {\it Proof.} By the triangle inequality and 
\eqref{CS} (with $\delta=1$), we have
\begin{align*}
|{\bf e}^\gamma_n(x){\bf e}^\gamma_n(y)-{\bf e}^\gamma_n(x'){\bf e}^\gamma_n(y')|\leq & |{\bf e}^\gamma_n(x)||{\bf e}^\gamma_n(y)-{\bf e}^\gamma_n(y')|+|{\bf e}^\gamma_n(y')|
|{\bf e}^\gamma_n(x)-{\bf e}^\gamma_n(x')|\\
\leq & C\lambda_{\gamma,n}\big(|{\bf e}^\gamma_n(x)-{\bf e}^\gamma_n(x')|+|{\bf e}^\gamma_n(y)-{\bf e}^\gamma_n(y')|\big).
\end{align*}
Now we estimate the quantity $|{\bf e}^\gamma_n(x)-{\bf e}^\gamma_n(x')|$. 
By using the eigenfunction relation \eqref{eigen} and then 
again the estimate \eqref{CS} with $\delta=1$, we obtain 
$$|{\bf e}^\gamma_n(x)-{\bf e}^\gamma_n(x')|=\lambda_{\gamma,n} \Big|  
\int_{\T}\big(G_\gamma(x,z)-G_\gamma(x',z)\big){\bf e}^\gamma_n(z)M_\gamma(dz)\Big|
\leq  C\lambda_{\gamma,n} ^2 h(x,x').$$
Summing up these relations over $n\geq 1$ we get
$$ |\p^\gamma_t(x,y)-\p^\gamma_t(x',y')|\leq C^2\sum_{n\geq 1}\lambda_{\gamma,n}^3e^{-\lambda_{\gamma,n} t}(h(x,x')+h(y,y')).$$
By the uniform convergence of the series 
$\sum_n\lambda_{\gamma,n}^{3}e^{-\lambda_{\gamma,n} t}$ for $t\geq t_0$, 
we conclude that the above estimate is uniform with respect to $t\geq t_0$. 
The same argument handles also the control over the
time dependence.\qed

\begin{corollary}
For all $x,y\in\T$ and $t>0$, we have $\p^\gamma_t(x,y)>0$.
\end{corollary}

\noindent {\it Proof.} From the spectral representation 
Theorem \ref{heat}, we have
$$\p^\gamma_t(x,x)=\frac{1}{M_\gamma(\T)} +\sum_{n\geq 1}
e^{-\lambda_{\gamma,n} t}{\bf e}^\gamma_n(x)^2>0.$$
 Then it suffices to adapt the proof of \cite[Proposition 5.1.10]{kigami}.\qed

\section{Upper bounds on the heat kernel}    
\label{sec-UB}

In this section, we state and prove our upper bound on the heat kernel.
We stick to the notations of section \ref{studyHK}. We begin with
a brief reminder of general techniques for deriving 
upper bounds on heat kernels associated to Dirichlet forms.
 
\subsection{Reminder on heat kernel estimates}
Here we will recall a weak form of Theorem  6.3 in \cite{grigor} (obtained by setting $h(t) = t^{1/\beta}$ and $F(x,y,h(t)) \equiv C(1+t^{-\alpha})$ there), which yields
upper bounds for heat kernels associated to Dirichlet forms. 
We consider a locally compact and separable metric space $(E,d)$ and 
$\mu$ a Radon measure on this metric space. We suppose that $\mu$ 
has full support, i.e. that $\mu(O)>0$ for every open set $O$.
 
\begin{lemma} \label{lemmeGrigorbis}
Let $\beta>1$ and $\alpha>0$. Consider the heat kernel $p_t$ associated to a conservative, local, regular Dirichlet form on $L^2(E,\mu)$ 
and let $\tau_{B(y,r)}$ denote the exit time of the associated Markov process 
from the ball $B(y,r)$.
Assume that \\
1) For all $x,y$ and $t >0$, we have \, $p_t(x,y)  \leq C\Big(\frac{1}{t^{\alpha}}+1 \Big).$\\

\noindent 2) There exists $\epsilon \in ]0,\frac{1}{2}[$ such that 
        $\underset{r \to 0}{\overline{\lim} }  
        \sup_{y \in E}  \Pb^y(\tau_{B(y,r)}  \leq r^\beta)  \leq \epsilon.$
        
\vspace{1mm}
\noindent Then, for all $t>0$ and $\mu$ almost all $x,y\in E$,
\begin{equation*}
p_t(x,y)  \leq   C'\Big(\frac{1}{t^{\alpha}}+1 \Big)\exp\Big(-C''   \left ( \frac{d(x,y)}{t^{1/\beta}} \right )^{\frac{\beta}{\beta-1}}\Big).
\end{equation*}
\end{lemma}

\subsection{The upper bound}

Set 
\begin{equation}\label{param}
\alpha =2\Big(1-\frac{\gamma}{2}\Big)^2 \quad \text{ and } \quad \forall u >0, \; \beta(u) =\Big( \frac{\gamma}{\sqrt{u}}+\sqrt{\frac{\gamma^2}{u}+2+\frac{\gamma^2}{2}}\,\,\Big)^2.
\end{equation}

Here is the main result of this section:
\begin{theorem}\label{upperG}
	 For each $\delta>0$, we set
	\begin{equation}\label{paramdelta}
	\alpha_\delta= \alpha-\delta, \; \; \beta_\delta =\beta(\alpha_\delta)+ \delta
	\end{equation} 
	Then, there exist  two random constants $c_1=c_1(X),c_2=c_2(X)>0$ such that
	$$\forall x,y\in\T,t>0,\quad \p^\gamma_t(x,y)\leq \frac{c_1}{t^{1+\delta}}\exp\Big(-c_2 \left ( \frac{d_{\T}(x,y)}{t^{1/ \beta_\delta}}  \right )^{\frac{\beta_\delta}{\beta_\delta-1}} \Big).$$
	\end{theorem}

\noindent
The upper bound of Theorem \ref{upperG} extend to the Liouville
Brownian Motion on the whole space; see Remark \ref{rem-genupper} below.

	\subsection{Proof of Theorem \ref{upperG}.} 

	As the Dirichlet form associated to the LBM is conservative, local and regular (see \cite[Section~2]{GRV2}), the proof is based on the following two lemmas and an application of Lemma \ref{lemmeGrigorbis}.

	\begin{lemma}\label{UB1}
	For each $\delta>0$, we can find $C_\delta=C_\delta(X)>0$ such that
	$$\forall t>0,\quad \sup_{x,y\in\T} \p^\gamma_t(x,y) \leq  
	C_\delta\big(1+ t^{-(1+\delta)}\big)  .$$
	\end{lemma}
	\begin{lemma}\label{UB2}
	Recall that $ \beta_\delta$ is defined by \eqref{paramdelta} and that 
	$$\tau_{B(x,r)}=\inf\{t>0; \LB_t \not \in B(x,r)\}.$$ 
	For each $\delta>0$, $\Pb^X$-almost surely we have
	$$ \underset{r \to 0}{\overline{\lim} }  \sup_{x\in \T} \Pmp{\Bvec}{x} [\tau_{B(x,r)}\leq r^{\beta_{\delta}}]
	\le\frac{1}{4}.$$
	\end{lemma}

	 \noindent {\it Proof of Lemma \ref{UB1}.} 
	By  \eqref{eq-insidethm3.3}, \eqref{eq-newsunday2} and
	 \eqref{eq:miss-kittin}, it suffices to consider $t<1$. 
	From the heat kernel representation \eqref{eq-insidethm3.3}, we have that
	 the mapping $t\mapsto \p_t^\gamma(x,x)$ is decreasing. 
	 Thus, since $t<1$,
	$$t^{1+\delta} \p_t^\gamma(x,x)\leq 2^{1+\delta} \int_{t/2}^{t}u^\delta 
	\p_u^\gamma(x,x)\,du\leq  2^{1+\delta}\int_{0}^{1}u^\delta 
	\p_u^\gamma(x,x)\,du.$$
	Combined with \eqref{eq-4.3},
	this shows that $\sup_{x\in\T}\sup_{t<1}t^{1+\delta} \p_t^\gamma(x,x)<+\infty$. Finally observe that the heat kernel representation 
	\eqref{eq-insidethm3.3}
	also yields by Cauchy-Schwarz's inequality  that
	$$t^{1+\delta} \p_t^\gamma(x,y)\leq \big(t^{1+\delta} \p_t^\gamma(x,x)\big)^{1/2}\big(t^{1+\delta} \p_t^\gamma(y,y)\big)^{1/2}\leq \sup_{x\in\T}\sup_{t<1}t^{1+\delta} \p_t^\gamma(x,x)<+\infty.$$
	The proof of the lemma is complete.\qed


	The proof of Lemma \ref{UB2} is based on a coupling argument. 
	We recall some preliminary observations.
	Consider two independent  Brownian  motions ${\Bvec},{\Wvec}$ on $\T$;
	possibly enlarging the probability space, we may and 
	will assume that ${\Bvec},{\Wvec}$ are defined on the same probability space with
	the 
	LCGF $X$. 
	Considering the torus $\T$ as $(\R/\Z)^2$, then for $i=1,2$
	we may speak of the components
	$B^i,W^i$ of the Brownian motions.
	We denote by $\Pmp{\Bvec,\Wvec}{x,y} $ the probability measure 
	$\Pmp{\Bvec}{x}\otimes\Pmp{\Wvec}{y}$. The following lemma is elementary; we 
	leave the proof to the reader.
	\begin{lemma}\label{lem:coupling}
	Introduce the successive coupling times $\tau_1,\tau_2$ of the components:
	$$\tau_1=\inf\{u>0;B^{1}_u=W^{1}_u\},\quad \tau_2=\inf\{u>\tau_1;B^{2}_u=W^{2}_u\}.$$ 
	Under $\Pmp{\Bvec,\Wvec}{x,y}$, the random process $\overline{{\Bvec}}$ defined by 
	$$\overline{{\Bvec}}_t=\left\{ 
	\begin{array}{lll}
	(W^{1}_t,W^{2}_t)  & \text{if}  &  t\leq \tau_1 \\
	(B^{1}_{t}, W^{2}_t)  & \text{if}  &  \tau_1<t\leq \tau_2\\
	(B^{1}_t,B^{2}_t)  & \text{if}  &  \tau_2<t .
	\end{array}
	\right.$$ is  a  Brownian motion on $\T$ starting from $y$, and coincides with ${\Bvec}$ for all times $t>\tau_2$. Furthermore, we have 
	\begin{align*}
	\forall\eta>0, \lim_{\epsilon\to 0} \sup_{x,y\in\T;|x-y|\leq \epsilon}\Pmp{\Bvec,\Wvec}{x,y}(\tau_2>\eta)\to 0, \quad \text{and }\quad
	\Pmp{\Bvec,\Wvec}{x,y}(\tau_2<\infty)=1 .
	\end{align*}
	\end{lemma}

	$\Pb^X$-a.s., we can associate to the Brownian motion $\overline{{\Bvec}}$ a PCAF,  denoted by $F(\overline{{\Bvec}},t)$ to distinguish it from $F$ associated to ${\Bvec}$,   with Revuz measure $M_\gamma$. Formally, 
	\begin{equation*}
		F(\overline{{\Bvec}},t)=
		\int_0^t  e^{\gamma X(\overline{{\Bvec}}_s)- 
		\frac{\gamma^2}{2}\E^X[X(\overline{{\Bvec}}_s)^2] }  ds.
	\end{equation*}
	 Introduce the first exit time $T_{B(x,r)}$ 
	 of the standard Brownian motion out of the ball $B(x,r)$ and note that
	under $\Pmp{\Bvec}{x}$,
	$$\tau_{B(x,r)}=F(T_{B(x,r)}).$$ 
	It is also plain to check that $\Pb^X$-a.s., for all $x,y\in\R^2$, under $\Pmp{\Bvec,\Wvec}{x,y}$, the marginal laws of $({\Bvec},F)$ and $(\overline{{\Bvec}},F(\overline{{\Bvec}},\cdot))$ respectively coincide with the law of $({\Bvec},F)$ under $\Pmp{\Bvec}{x}$ and $\Pmp{\Bvec}{y}$. 

	\vspace{2mm}
	Now, we state three quantitative results about the behaviour of the PCAF $F$ and the measure $M_\gamma$. Recall that $\zeta(q)=(2+\frac{\gamma^2}{2})q -\frac{\gamma^2}{2}q^2$.
	\begin{lemma}\label{UB3}
	For each $q>0$, there exists a  constant $C_q$ 
	such that for all $r\in]0,1]$ and $x\in\T$
	$$\E_x[F(T_{B(x,r)})^{-q}]\leq C_q r^{\zeta(-q)}.$$
	\end{lemma}

	\noindent {\it Proof.} See \cite[Prop. 2.12]{GRV1}.\qed

	\begin{lemma}\label{UB5}
	Set $\alpha=2(1-\frac{\gamma}{2})^2$. For each   $\delta>0$,  $\Pb^X$-almost surely, we have
	$$ \sup_{x\in \T}\sup_{r\in]0,1]}r^{-(\alpha-\delta)} \int_{B(x,r)}\ln\frac{1}{d_\T(x,y)}M_\gamma(dy)<+\infty.$$
	\end{lemma}

	\noindent {\it Proof.} The lemma  follows directly from \eqref{holderM}.\qed

	\begin{lemma}\label{UB4}
	Fix $\delta>0$ and set $\alpha=2(1-\frac{\gamma}{2})^2$. Then, 
	$\Pb^X$-almost surely, there exists a random constant $D_{\gamma,\delta}=
	D_{\gamma,\delta}(X)>0$ such that 
	$$ \sup_{x\in \T}\sup_{r\in]0,1]}r^{-(\alpha-\delta)} 
	\Emp{\Bvec}{x}[F(T_{B(x,r)})]\leq D_{\gamma,\delta}.$$
	\end{lemma}

	\noindent {\it Proof.} First observe that 
	$$\Emp{\Bvec}{x}[F(T_{B(x,r)})]=\int_{B(x,r)}G_{B(x,r)}(x,y)M_\gamma(dy),$$
	where $G_{B(x,r)}(x,y)$ stands for the Green function in the ball $B(x,r)$ killed upon touching the boundary $\partial B(x,r)$. Furthermore $G_{B(x,r)}(x,y)=\frac{1}{\pi}\ln\frac{r}{d_\T(x,y)}\leq \frac{1}{\pi}\ln\frac{1}{d_\T(x,y)}$. Thus we have
	$$\Emp{\Bvec}{x}[F(T_{B(x,r)})]\leq \int_{B(x,r)} \frac{1}{\pi}\ln\frac{1}{d_\T(x,y)}M_\gamma(dy).$$
	Then it suffices to  apply Lemma \ref{UB5}.\qed

	\medskip
	We are now in a position to prove Lemma \ref{UB2}.

	\medskip

	\noindent {\it Proof of Lemma \ref{UB2}.} It suffices to treat the case where the supremum in $r$ runs over $r<1$. Recall that under $\Pmp{\Bvec}{x}$
	$$\tau_{B(x,r)}=F(T_{B(x,r)}).$$

	The first step of the proof is to relate the 
	behaviour of the quantity $\Pmp{\Bvec}{y}(\tau_{B(y,r)}\leq t)$ 
	to the behaviour of $\Pmp{\Bvec}{x}(\tau_{B(x,r)}\leq t)$ for 
	all those $y$ that are close enough to $x$. 
	This is provided by the following claim:
	there exists $\ell>0$ deterministic 
	such that $\forall x\in\T,\forall r<r^*(\ell),\forall t>0$,
	\begin{equation}\label{smallball}
	  \sup_{y,|y-x|\leq r^{ \beta_\delta/\alpha_\delta}}\Pmp{\Bvec}{y}(\tau_{B(y,r)}\leq t)\leq  \frac{1}{4}+\Pmp{\Bvec}{x}(\tau_{B(x,r/2)}\leq 2t)+\frac{4}{t}\sup_{y\in\T}\Emp{\Bvec}{y}[F(T_{B(y,\ell r^{ \beta_\delta/\alpha_\delta})})].
	\end{equation}
	We provide the (coupling based)
	proof of \eqref{smallball} at the end of the proof of the lemma.

	In the next step, we take $r:=r_n=\frac{1}{2^n}$.
	Once \eqref{smallball} is established, we 
	consider  a covering of the torus $\T$ with $N_n$ balls 
	$(B_k^n)_{1\leq k\leq N_n}$ of radius $r_n^{ \beta_\delta/\alpha_\delta}$
	and centers
	$(x^n_k)_{1\leq k\leq N_n}$;
	we can find such a covering with $N_n\leq 
	Cr_n^{-2 \beta_\delta/\alpha_\delta}$ for some deterministic
	constant $C>0$ independent of $n$. We will 
	establish that there exists  a deterministic 
	$\epsilon>0$ small enough such that, $\Pb^X$-almost surely, there exists 
	a random $n_0=n_0(X)$  such that $\forall n\geq n_0$,
	\begin{equation}\label{rational}
	\sup_{1\leq k\leq N_n}
	\Pmp{\Bvec}{x^n_k}(\tau_{B(x^n_k,2^{-n})}\leq 
	2^{-n \beta_\delta})\leq 2^{-n\epsilon}.
	\end{equation}

	Finally, by combining 
	\eqref{smallball},\eqref{rational} and
	Lemma \ref{UB4}, we deduce that there exists a random $r_0=r_0(X)$ so that
	$\forall r<r_0$, with $n = \lfloor \log_2 \frac 1 r \rfloor$,
	\begin{align*}
	 \sup_{y\in\T}\Pmp{\Bvec}{y}(\tau_{B(y,4r)}  \leq \tfrac 1 2 r^{\beta_\delta}) 
	 \leq & \frac{1}{4}+  \sup_{1\leq k\leq N_n}
	\Pmp{\Bvec}{x^n_k}(\tau_{B(x^n_k,2^{-n})}\leq 
	2^{-n \beta_\delta}) +\frac{4}{r^{\beta_\delta}}\sup_{y\in\T}\Emp{\Bvec}{y}[F(T_{B(y,\ell r^{ \beta_\delta/\alpha_\delta})})]\\
	 \leq & \frac{1}{4}+ (2r)^\epsilon+\frac{4}{r^{\beta_\delta}}D_{\gamma,\delta'}\ell^{\alpha-\delta'} r^{\beta_\delta \frac{ \alpha-\delta'}{\alpha-\delta} },
	\end{align*}
	where we have chosen $\delta'<\delta$. We deduce that 
	$$\limsup_{r\to 0} \sup_{y\in\T}\Pmp{\Bvec}{y}(\tau_{B(y,4r)}  \leq \tfrac 1 2 r^{\beta_\delta}) \leq 1/4.$$ This completes the proof of
	Lemma \ref{UB2},
	provided that we can prove \eqref{smallball} and \eqref{rational}.

\medskip
We begin with the proof of
\eqref{rational}. By using in turn the Markov inequality and Lemma \ref{UB3}, we have for all $p>0$
\begin{align}
\Pb^X\Big(\max_{1\leq k \leq N_n}&\Pmp{\Bvec}{x^n_k}(\tau_{B(x^n_k,2^{-n})}\leq 2^{-n \beta_\delta}) \geq 2^{-n\epsilon}\Big)\nonumber\\
\leq &2^{n\epsilon}\E^X\Big[ \max_{1\leq k \leq N_n} \Pmp{\Bvec}{x^n_k}(\tau_{B(x^n_k,2^{-n})}\leq 2^{-n \beta_\delta}) \Big]\nonumber\\
\leq &2^{n\epsilon}\E^X\Big[\max_{1\leq k \leq N_n}2^{-n p\beta_\delta}\Emp{\Bvec}{x^n_k}[F(T_{B(x^n_k,2^{-n})})^{-p} ] \Big]\nonumber\\
\leq & 
C2^{n\epsilon-np\beta_\delta}\sum_{1\leq k \leq N_n}\E_{x^n_k}[F(T_{B(x^n_k,2^{-n})})^{-p} ] \nonumber \\
\leq & 
CC_p2^{n\epsilon-np\beta_\delta}2^{2n \beta_\delta/\alpha_\delta}2^{-n\zeta(-p)}.\label{quant}
\end{align}
Consider the function
$$f(p)=-p\beta_\delta+2 \beta_\delta/\alpha_\delta-\zeta(-p)=\frac{\gamma^2}{2}p^2+(2+\frac{\gamma^2}{2}-\beta_\delta)p+2 \beta_\delta/\alpha_\delta.$$
By the choice of $\alpha$ and $\beta$ in \eqref{param}, there exists $p>0$ such that $f(p)<0$ and fix $\epsilon=-f(p)/2>0$. 
Indeed, the minimum of the 
function $f$ is attained for 
$p_\star= (\beta_\delta-2-\frac{\gamma^2}{2})/\gamma^2 > 0$ and equals
$$f(p_\star)= 2 \beta_\delta/\alpha_\delta -
\frac{(\beta_\delta-2-\frac{\gamma^2}{2})^2}{2 \gamma^2}.$$ 
The last expression
is negative by \eqref{paramdelta}.
Finally, we use the Borel-Cantelli Lemma
in order to complete the proof of  \eqref{rational}.

\medskip
We turn to the proof of
\eqref{smallball}.  We fix $x\in\T$ and consider 
$y\in \T$ such that $|y-x|\leq r^{\beta_\delta/\alpha_\delta}$. 
We will use Lemma \ref{lem:coupling} and the notation
introduced there.
 We further introduce the first exit times $T^{\Bvec}_{B(z,r)}$ and $T^{\overline{{\Bvec}}}_{B(z,r)}$ of the Brownian motions ${\Bvec}$ and $ \overline{{\Bvec}}$ out of the ball $B(z,r)$. We have
\begin{align*}
\Pmp{\Bvec}{y}(\tau_{B(y,r)}\leq t)=&\Pmp{\Bvec}{y}(F(T_{B(y,r)})\leq t)\\
=&\Pmp{\Bvec,\Wvec}{x,y}(F(\overline{B},T^{\overline{{\Bvec}}}_{B(y,r)})\leq t)\\
=&\Pmp{\Bvec,\Wvec}{x,y}\Big(F(\overline{{\Bvec}},T^{\overline{{\Bvec}}}_{B(y,r)})\leq t,\tau_2\leq \min(T^{\Bvec}_{B(x,\ell r^{\beta_\delta/\alpha_\delta})},T^{\overline{{\Bvec}}}_{B(y,\ell r^{\beta_\delta/\alpha_\delta})})\Big)\\
&+\Pmp{\Bvec,\Wvec}{x,y}\Big( \tau_2> \min(T^{\Bvec}_{B(x, \ell r^{\beta_\delta/\alpha_\delta})},T^{\overline{{\Bvec}}}_{B(y,\ell r^{\beta_\delta/\alpha_\delta})})\Big).
\end{align*}

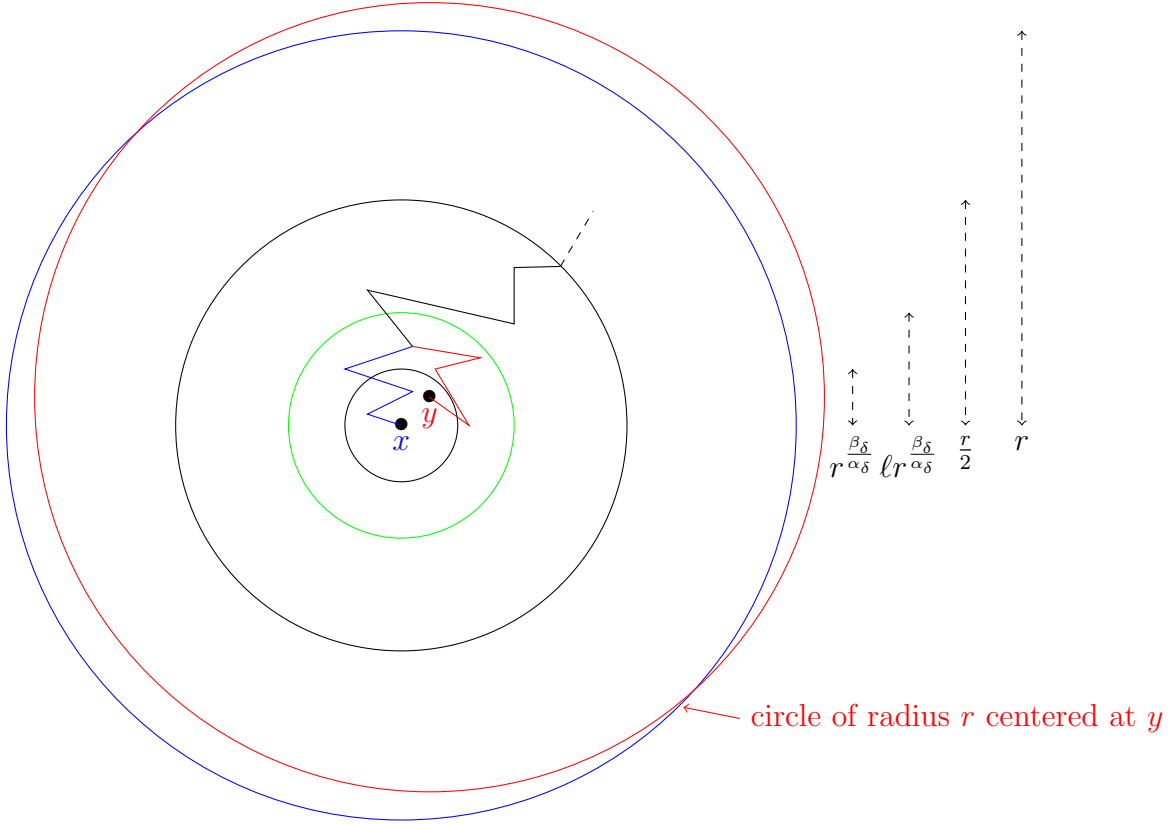
\begin{figure}[h] 
\begin{center}
\begin{tikzpicture}[scale=1.5] 
\draw (0,0) node {$\bullet$};
\draw [blue] (0,0) node[below] {$x$};
\draw (0.25,0.25) node {$\bullet$};
\draw [red] (0.25,0.25) node[below] {$y$};
\draw [dashed,<->] (4,0.5) -- (4,0) node[below] {$r^{\frac{\beta_\delta}{\alpha_\delta}}$};
\draw [dashed,<->] (4.5,1) -- (4.5,0) node[below] {$\ell r^{\frac{\beta_\delta}{\alpha_\delta}}$};
\draw [dashed,<->] (5,2) -- (5,0) node[below] {$\frac{r}{2} $};
\draw [dashed,<->] (5.5,3.5) -- (5.5,0) node[below] {$r $};
\draw (0,0) circle (0.5) ;
\draw [green] (0,0) circle (1) ;
\draw (0,0) circle (2) ;
\draw [blue] (0,0) circle (3.5) ;
\draw [red] (0.25,0.25) circle (3.5);
\draw [blue] (0,0)--(-0.3,0.1)--(0.1,0.3)--(-0.5,0.5)--(0.1,0.7);
\draw [red] (0.25,0.25)--(0.6,0)--(0.3,0.5)--(0.7,0.6)--(0.1,0.7);
\draw   (0.1,0.7)--(-0.3,1.2)--(1,0.9)--(1,1.4)--(1.41,1.41);
\draw [dashed]  (1.41,1.41)--(1.7,1.9);
\draw [red,<-] (2.5,-2.5)--(3,-2.6) node[right] {circle of radius $r$ centered at $y$};
\end{tikzpicture}
\end{center}
\caption{Illustration of the coupling: the blue Brownian motion starts from $x$ whereas the red one start from $y$. Once they have been coupled, their paths is drawn in black. They are forced to be coupled before they hit the green circle. Notice that the circle centered at $x$ of radius $r/2$ is contained in the intersection of the inner parts of the blue and red circles.}
\label{fig:couplig}
\end{figure}

By using the scaling relations and the symmetries 
(translation invariance and isotropy) of Brownian motion, we have that
$$\Pmp{\Bvec,\Wvec}{x,y}\Big( \tau_2> 
\min(T^{\Bvec}_{B(x,\ell
        r^{\beta_\delta/\alpha_\delta})},T^{\overline{{\Bvec}}}_{B(y,\ell 
                r^{\beta_\delta/\alpha_\delta})})\Big)\leq 
                \Pmp{\Bvec,\Wvec}{0,z}\Big( \tau_2> 
                \min(T^{\Bvec}_{B(0,\ell)},T^{\overline{{\Bvec}}}_{B(y,\ell)})\Big)$$
where 
$z$ is any point of the torus at distance $1$ of $0$ 
(the above quantity is independent of the choice of such a $z$). Now we   choose $\ell$ large enough so as to make the right hand side
less than $1/4$. Then on the event 
$\{\tau_2\leq \min(T^{\Bvec}_{B(x,\ell r^{\beta_\delta/\alpha_\delta})},
T^{\overline{{\Bvec}}}_{B(y,\ell r^{\beta_\delta/\alpha_\delta})})\}$, 
the two Brownian motions are coupled before they both leave the 
ball $B(x,r/2)$, as long as $r<(2\ell)^{1-\beta_\delta/\alpha_\delta}$.
Furthermore, on this event, the paths of the Brownian motions 
${\Bvec},\overline{{\Bvec}}$ coincide from $\tau_2$ until $T_{B(x,r/2)}$. 
Therefore under $\Pmp{\Bvec,\Wvec}{x,y}$ and on the event  $\{\tau_2\leq \min(T^{\Bvec}_{B(x,\ell r^{\beta_\delta/\alpha_\delta})},T^{\overline{{\Bvec}}}_{B(y,\ell r^{\beta_\delta/\alpha_\delta})})\}$ we have
\begin{align*}
F(\overline{{\Bvec}},T^{\overline{{\Bvec}}}_{B(x,r/2)})=&F(\overline{{\Bvec}},T^{\overline{{\Bvec}}}_{B(x,r/2)})-F(\overline{{\Bvec}},\tau_2)+F(\overline{{\Bvec}},\tau_2)\\
=&F({\Bvec},T^{\Bvec}_{B(x,r/2)})-F({\Bvec},\tau_2)+F(\overline{{\Bvec}},\tau_2).
\end{align*}
Hence we get
\begin{align*}
\Pmp{\Bvec,\Wvec}{x,y}&\Big(F(\overline{{\Bvec}},T^{\overline{{\Bvec}}}_{B(y,r)})\leq t,\tau_2\leq \min(T^{\Bvec}_{B(x,\ell r^{\beta_\delta/\alpha_\delta})},T^{\overline{{\Bvec}}}_{B(y,\ell r^{\beta_\delta/\alpha_\delta})})\Big)\\
\leq & \Pmp{\Bvec,\Wvec}{x,y} \Big(F(\overline{{\Bvec}},T^{\overline{{\Bvec}}}_{B(x,r/2)})\leq t,\tau_2\leq \min(T^{\Bvec}_{B(x,\ell r^{\beta_\delta/\alpha_\delta})},T^{\overline{{\Bvec}}}_{B(y,\ell r^{\beta_\delta/\alpha_\delta})})\Big)\\
\leq & \Pmp{\Bvec,\Wvec}{x,y} \Big(F({\Bvec},T^{\Bvec}_{B(x,r/2)})\leq 2t,\tau_2\leq \min(T^{\Bvec}_{B(x,\ell r^{\beta_\delta/\alpha_\delta})},T^{\overline{{\Bvec}}}_{B(y,\ell r^{\beta_\delta/\alpha_\delta})}),F({\Bvec},\tau_2)\leq t/2\Big)\\
&+\Pmp{\Bvec,\Wvec}{x,y} \Big( \tau_2\leq \min(T^{\Bvec}_{B(x,\ell r^{\beta_\delta/\alpha_\delta})},T^{\overline{{\Bvec}}}_{B(y,\ell r^{\beta_\delta/\alpha_\delta})}),F({\Bvec},\tau_2)>t/2\Big)\\
\leq & \Pmp{\Bvec}{x} \big(F(T_{B(x,r/2)})\leq 2t\big) +\Pmp{\Bvec,\Wvec}{x,y} \big( F({\Bvec},T^{\Bvec}_{B(x,\ell r^{\beta_\delta/\alpha_\delta})})>t/2\big)\\
\leq & \Pmp{\Bvec}{x} \big(F(T_{B(x,r/2)})\leq 2t\big)+\frac{2}{t}\sup_{y\in \T}\Emp{\Bvec}{y} \big( F(T_{B(y,\ell r^{\beta_\delta/\alpha_\delta})})\big).
\end{align*}
The proof is complete. \qed

\begin{remark}
\label{rem-genupper}
  It is straightforward to check that the heat kernel estimate in 
  Theorem \ref{upperG} extends to the Liouville Brownian motion on 
  the whole space, in the following form. Let
  $\bar\p^\gamma_t(x,y)$ denote the whole space heat kernel. Then, for every
  $R\geq 0$ there exists random constants $c_1(R,X),c_2(R,X)>0$ so that
  \begin{equation}
    \label{eq-may11}
    \sup_{t>0, |x|,|y|\leq R}
  \bar \p^\gamma_t(x,y)\leq \frac{c_1}{t^{1+\delta}}\exp\Big(-c_2 
  \left ( \frac{|x-y|}{t^{1/ \beta_\delta}}  
  \right )^{\frac{\beta_\delta}{\beta_\delta-1}} \Big).
\end{equation}
To see \eqref{eq-may11}, we may and will assume by scaling that $R<1/4$.
Let $\hat \p^\gamma_t$ denote the heat kernel of the LBM killed upon exiting
$B(0,3/4)$; we have that $\hat \p_t^\gamma \leq \p_t^\gamma\leq W(t,x,y)$
where $W(t,x,y)$ is the upper bound on $\p_t^\gamma$ from Theorem
\ref{upperG}.
Using Lemma \ref{UB2}, conditioned on $X$,
the number of excursions between $\partial B(0,3/4)$
and $\partial B(0,1/2)$ before time $1$ 
is dominated by a geometric random variable of finite
mean $M(X)$. Then, using the Markov property and $t\leq 1$,
\begin{eqnarray*}
  \bar \p^\gamma_t(x,y)&\leq &
\hat  \p^\gamma_t(x,y)+M(X) \sup_{s\leq t, z\in \partial B(0,1/2)} 
\hat \p^\gamma_s(z,y)\\
&\leq& 
 \p^\gamma_t(x,y)+M(X) \sup_{s\leq t, z\in \partial B(0,1/2)} 
\p^\gamma_s(z,y)\leq C(X)W(t,x,y)\,,
\end{eqnarray*}
as claimed.
\end{remark}

\section{Lower bounds on the heat kernel}
\label{sec-LB}

The goal of this section is to prove the following two theorems:
\begin{theorem}\label{th:LB-heat-kernel}
Fix $x \ne y$. For all $\eta>0$, there exists some random variable $T_0=T_0(x,y,\eta)$ such that for all $t\leq T_0$,
\begin{align*}   
\p^\gamma_t(x,y) \geq \exp\Big(- t^{-{\frac{1}{1+\gamma^2/4-\eta}}}\Big),\quad \text{$\Pb^X$-a.s.}
\end{align*}
\end{theorem}

\begin{theorem}\label{th:LB-heat-kernel-Mgamma}
Conditioned on the Gaussian field $X$, let $x$,$y$ be sampled according to the measure $M_\gamma(\T)^{-1}M_\gamma$. For all $\eta>0$, there exists some random variable $T_0$, such that for all $t\leq T_0$,
\begin{align*}   
\p^\gamma_t(x,y) \geq \exp\Big(- t^{-{\frac{1}{\nu(\gamma)-\eta}}}\Big),\quad \text{$\Pb^X$-a.s.,}
\end{align*}
where
\begin{equation}
\nu(\gamma) = 
\begin{cases}
1+\frac{\gamma^2}{4} & \gamma^2 \in [0,8/3]\\
1+\gamma^2 - \frac{\gamma^2}{4}\left(1-\frac{\gamma^2}{4}\right)^{-1} & \gamma^2 \in (8/3,3]\\
4-\gamma^2 & \gamma^2 \in (3,4).
\end{cases}
\label{eq:nu_gamma}
\end{equation}
\end{theorem}


The remainder of the section is organized as follows: in Section~\ref{lowerboundsection}, we first give a lower bound on the resolvent. The general strategy for this is explained at the beginning of the section. Section~\ref{sec:harnack_lower} then introduces some Harnack inequalities which are used to upgrade the resolvent bounds to bounds on the heat kernel, yielding Theorem~\ref{th:LB-heat-kernel}. In Section~\ref{sec:lower_bound_some_gamma}, we then show that the results from the two previous sections can be applied with few changes to prove Theorem~\ref{th:LB-heat-kernel-Mgamma} for $\gamma \le 4/3$. In Section~\ref{sec:lower_bound_all_gamma}, we give a refined strategy which allows to treat all $\gamma < 2$.


%
\subsection{Lower bound on the resolvent} \label{lowerboundsection}

 
\eqref{eq:resolvent}. 
For $y\in\T$ and $r>0$, let  $B_r(y)$ denote the ball of radius $r$ around $y$. 
\begin{theorem}\label{lowerboundreso}
Fix $x \ne y$. There exists a numerical constant $\beta>0$, such that for all $\eta>0$, there exists some random variable $\Lambda_0=\Lambda_0(x,y,\eta)$ such that for all $\lambda \geq \Lambda_0$,
\begin{align*}   
\inf_{w\in B_{\lambda^{-\beta}}(x)}
\inf_{z\in B_{\lambda^{-\beta}}(y)}
\r^\gamma_\lambda(w,z) \geq \exp\Big(- \lambda^{\frac{1}{2+\gamma^2/4-\eta}}\Big).
\end{align*}
\end{theorem}

Theorem~\ref{lowerboundreso} implies in particular that the resolvent is superdiffusive, i.e. decreases strictly slower than $e^{- \sqrt{\lambda}}$ as $\lambda$ goes to infinity. 

We will prove in fact a slightly 
stronger result, namely the following theorem: 
%
\begin{theorem}
\label{th:lower_bound}
Fix $x \ne y$. For all $\eta>0$, there exists some random variable $T_0 = T_0(x,y,\eta)>0$ and some numerical constant $c>0$, such that for $t\leq T_0$ and $\lambda > 0$,
\[
\inf_{w\in B_{t/2}(x)} 
\inf_{z\in B_{t/2}(y)} 
\Ebr \Bvec w z {t} 
[e^{-\lambda F(t)}] \geq e^{-c(\lambda t^{1+\gamma^2/4-\eta} +t^{-1})}.
\]
\end{theorem}
Notice that Theorem \ref{lowerboundreso} directly results from Theorem~\ref{th:lower_bound} by the obvious relation $$\r^\gamma_\lambda(x,y)\geq \int_0^{T_0}\Ebr \Bvec x y {T_0} [e^{-\lambda F(t)}] p_t(x,y)\,dt$$ 
in combination with the elementary 
estimate
 Lemma \ref{laplace} and  the fact that for some $c>0$, for all $x,y\in\T$, $p_t(x,y) \geq e^{-c/t}$ for $t\leq 1$.

\begin{figure}[!h]%
\vspace{1cm}
\def\svgwidth{\textwidth}
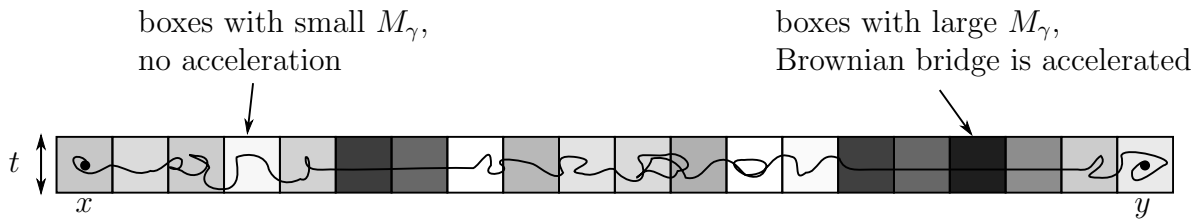
\caption{The ``strategy'' of the Brownian bridge in the statement of Theorem~\ref{th:lower_bound} for minimizing the functional $F(t)$.}%
\label{fig:acceleration}%
\end{figure}

We first give the heuristic ideas behind the proof. In order to bound $\Ebr \Bvec x y {t} [e^{-\lambda F(t)}]$ from below, the basic principle is to exhibit a strategy for the Brownian bridge whose probability is not less than $e^{-1/t}$ (or $e^{-c/t}$ for some constant $c>0$) and such that the typical  value of $F(t)$ is small. The strategy that we give is very simple: we first force the Brownian motion to follow basically a straight line from $x$ to $y$. Indeed, the probability for the Brownian bridge to stay in a tube of width $t$ around this straight line is of order $e^{-1/t}$. We then discretize this tube into small squares $S_0,\ldots,S_n$ of side length $t$ and consider the values of $M_\gamma(S_k)$, $k=0,\ldots,n$. The strategy is now to accelerate the Brownian bridge as soon as it enters boxes with large values of $M_\gamma$ (see Figure~\ref{fig:acceleration}).  The multifractal analysis of $M_\gamma$ quantifies this acceleration: say that a box $S_k$ is $\delta$-thick, if $M_\gamma(S_k)\approx t^{2+\gamma^2/2-\delta\gamma}$. Denote by $T_\delta$ the time the Brownian bridge spends in $\delta$-thick boxes. The contribution $F_\delta$ of these boxes to the functional $F(t)$ is then approximately
\[
 F_\delta = T_\delta \times t^{\gamma^2/2-\delta\gamma}.
\]
We can express this quantity differently: suppose we give the additional drift $v_\delta$ to the Brownian bridge in the $\delta$-thick boxes. By standard results for one-dimensional log-correlated fields, the number of $\delta$-thick boxes is of the order $t^{\delta^2/2-1}$ (there are no $\delta$-thick boxes for $|\delta| > \sqrt 2$), i.e.\ their total width equals $t^{\delta^2/2}$. The above quantity then takes on the form
\begin{equation}
\label{eq:F_delta}
 F_\delta = t^{\gamma^2/2-\delta\gamma+\delta^2/2}/v_\delta.
\end{equation}
We now try to maximize $v_\delta$, under the constraint that the cost of this acceleration is at most of order $1/t$ (i.e., the probability of such an event is at least of order $e^{-1/t}$). Since the width of the area in which we accelerate is $t^{\delta^2/2}$, standard large deviation estimates yield that the cost is
\[
 t^{\delta^2/2} v_\delta,
\]
which yields a maximal $v_\delta$ of $t^{-1-\delta^2/2}$ under the constraint. Plugging this into \eqref{eq:F_delta} gives
\[
 F_\delta = t^{1+\gamma^2/2-\delta\gamma+\delta^2} = t^{1+(\delta-\gamma/2)^2+\gamma^2/4}
\]
This quantity is maximized for $\delta = \gamma/2$, where it is 
\[
 \max_{|\delta| \leq \sqrt 2} F_\delta  = F_{\gamma/2} = t^{1+\gamma^2/4}.
\]
This is exactly the term occurring in the statement of Theorem~\ref{th:lower_bound} (an additional fudge factor $\eta$ is introduced there).

In reality, we cannot accelerate by the maximal $v_\delta = t^{-1-\delta^2/2}$ in the $\delta$-thick boxes for \emph{every} $\delta$. In order to get a simple formula, we therefore replace the quadratic polynomial $\delta^2/2$ by its tangent line at $\delta=\gamma/2$, i.e., we set 
\[
v_\delta = t^{-1 - \gamma^2/8 - (\delta-\gamma/2)\gamma/2},\quad\text{or},\quad v(S_k) =  t^{-2-\gamma^2/8} \sqrt{M_\gamma(S_k)}.
\]
This is precisely defined below.

We stress again that the largest contribution to $F(t)$ comes from the $\gamma/2$-thick boxes. This is a first indication that the distances in Liouville quantum gravity are \emph{not} determined by the $\gamma$-thick points of the underlying log-correlated Gaussian field, which is to be contrasted with the fact that the measure $M_\gamma$ is in fact ``supported'' on the $\gamma$-thick points, a fact which is made precise in \cite{cf:Kah} (see also \cite{review}).

We now get to the details of the proof of Theorem~\ref{th:lower_bound}. In order to simplify notation, we will assume in this section that the torus $\T$ is parametrized as $[-1,2]^2$ with identification of the opposite sides of the box, and that the points $x$ and $y$ are $x=(0,0)$ and $y=(1,0)$. The case of general $x$ and $y$ is no different. 
We will also write $\0:=(0,0)$ and $\1:=(1,0)$.

\subsubsection*{Step 1: Preliminaries on the random field}

We consider a fixed $t>0$. This parameter will take small values in the following. For convenience, we suppose that $n = 1/(2t)\in\N$. For $k=0,\ldots n$, let 
\begin{equation}
x_k = 2kt,\quad S_k = [x_k -t,x_k+t]\times [-t,t] \subset \T.
\end{equation}
We then define a family $W_0,\ldots,W_n\geq 0$ of random variables by the following formula for $k \not \in \lbrace 0,1,n-1,n \rbrace$   

\[
 W'_k = t^{-2-\gamma^2/8+\eta/2}\sqrt{M_\gamma(S_k)},\quad W_k = \begin{cases}
                                                           W'_k,\quad\text{ if }W'_k \geq 1/t \\
                                                           0,\quad\text{ otherwise.}
                                                          \end{cases}
\]
and $W_k=0$ for $k \in \lbrace 0,1,n-1,n \rbrace$. 

The speed-up strategy we will define later assigns the Brownian bridge a speed of roughly $W_k$ in the box $S_k$, for each $k\in\{0,\ldots,n\}$. In this section, we first prove some properties of these random variables, namely (A1-5) below. Their meaning is as follows: The quantity appearing on the LHS of (A1) is the cost of the speed-up strategy, which we want to be no bigger than $1/t$, as explained above. The LHS of (A2) bounds the expected value of the functional $F(t)$ under the speed-up strategy. Assumption (A3) says that there are many ``good'' boxes around the point $(1/2,0)$; this will allow us to split the Brownian bridge into two parts. Assumption (A5) ensures that the Brownian bridge is not accelerated in the first and the last box; the contribution to the functional $F(t)$ of these boxes is then bounded by the LHS in (A4).

Here is the precise statement: for each $\eta ,\eps>0$, there is a random variable $T_0>0$, such that the following holds for $t\leq T_0$:
\begin{enumerate}
 \item[(A1)] $t \sum_{k=0}^n W_k \leq \frac 1 t$,
 \item[(A2)] $\frac 1 t \sum_{k=0}^n (\frac 1 t + W_k)^{-1} M_\gamma(S_k) \leq  t^{1+\gamma^2/4-\eta}$
 \item[(A3)] $\#\{k\in\{0,\ldots,n\}:S_k\cap[1/2-\sqrt t,1/2+\sqrt t]\times [-t,t]\ne \emptyset \text{ and }[W_k \ne 0\text{ or } M_\gamma(S_k) \geq t^{2+\gamma^2/4-\eta}]\} \leq \eps/(6\sqrt t)$
 \item[(A4)] $\sup_{y\in B(x_0,t)}\int_{B(y,t)} \log_+\frac{t}{|x-y|}\,M_\gamma(dx) +\sup_{y\in B(x_n,t)} \int_{B(y,t)} \log_+\frac{t}{|x-y|}\,M_\gamma(dx) \leq t^{1+\gamma^2/4-\eta}.$
 \item[(A5)] $W_0=W_1=W_{n-1}=W_n=0$.
\end{enumerate}

Let us show (A1-5).
First, we will use the following lemma:

\begin{lemma}\label{lemmeMulti}
Let $\eta'>0$. Then there exists some random constant $C>0$ such that for all $t \leq 1$   we have
\begin{equation*}
\sum_{k=0}^n t^{-\gamma^2/8+\eta'}\sqrt{M_\gamma(S_k)}  \leq C
\end{equation*}

\end{lemma}

\proof

Set $D_t=  \sum_{k=0}^n t^{-\gamma^2/8+\eta'}\sqrt{M_\gamma(S_k)}$. We have
\begin{align*}
  \E^X  \left [  D_t  \right  ]   
& \leq C t^{-1-\gamma^2/8+\eta'}  \E^X  [   \sqrt{M_\gamma(S_k)}    ]  \\
& \leq C t^{-1-\gamma^2/8+\eta'+\zeta(1/2) }  \quad (\text{use }\eqref{powerlaw})\\
& \leq C t^{\eta' }
\end{align*}
where $C$ is some deterministic constant. Hence, we have almost surely $\sum_{N=0}^{\infty}  D_{2^{-N}} < \infty $; in particular, the variable $D_t$ converges almost surely to $0$ as $t$ goes to $0$. This is clear for a  dyadic $t$ and results for all $t$ by the following property: if $\frac{1}{2^{N+1}}< t < \frac{1}{2^{N}}$ for some integer $N$, then we have $D_t \leq c D_{1/2^N}$ for some numerical constant $c>0$.

\qed

Now, it is easy to check that our family $(W_k)_k$ satisfies (A1) and (A2). Indeed, we have by Lemma \ref{lemmeMulti} (with $\eta'=\eta/4$) the existence of some random constant $C>0$ such that 
\begin{equation*}
\sum_{k=0}^n W_k  \leq \sum_{k=0}^n t^{-2-\gamma^2/8+\eta/2}\sqrt{M_\gamma(S_k)} \leq C/t^{2-\eta/4}, 
\end{equation*}
which implies (A1).
Notice that by definition of $W_k$, we have $\frac 1 t + W_k  \geq  t^{-2-\gamma^2/8+\eta/2}\sqrt{M_\gamma(S_k)} $. 
Therefore, we have again by Lemma \ref{lemmeMulti} (with $\eta'=\eta/2$) the existence of some random constant $C>0$ such that  
\begin{equation*}
\sum_{k=0}^n (\frac 1 t + W_k)^{-1} M_\gamma(S_k)/t^2     \leq \sum_{k=0}^n t^{\gamma^2/8-\eta/2} \sqrt{M_\gamma(S_k)} \leq C t^{\gamma^2/4-\eta},
\end{equation*}
which implies (A2).

For property (A3), first notice that $W_k \not = 0$ is equivalent to $M_\gamma(S_k) \geq t^{2+\gamma^2/4-\eta}$. Therefore, it is a straightforward consequence of the following multifractal analysis lemma:

\begin{lemma}\label{lemmeA3}
Let $a \in ]0,1]$. Set $\mathcal S = \{k\in\{0,\ldots,n\}:S_k\cap[1/2-\sqrt t,1/2+\sqrt t]\times [-t,t]\ne \emptyset \}$. Then for all $\delta>0$, $\Pb^X$-almost surely, there exists some random constant $C$ such that for all $t\in(0,1],$  
\begin{align}
 \#\{k\in \mathcal S: M_\gamma(S_k) \geq t^{2+\gamma^2/2-a \gamma}   \} \leq C t^{-1/2(1-a^2+\delta)}.\label{multi2}
\end{align}

\end{lemma}

\proof
Note that $\#\mathcal S \le Ct^{-1/2}$. We have by Markov's inequality
\begin{align*}
& \Pb^X (    \#\{k\in \mathcal S: M_\gamma(S_k) \geq  t^{2+\gamma^2/2-a \gamma}   \} \geq  t^{-1/2(1-a^2+\delta)}   ) \\
& \leq  t^{1/2(1-a^2+\delta/2)}  \E^X (    \#\{k\in \mathcal S: M_\gamma(S_k) \geq t^{2+\gamma^2/2-a \gamma}   \}   ) \\
& \leq C t^{-a^2/2+\delta/2} \Pb^X(M_\gamma(S_0) \geq t^{2+\gamma^2/2-a \gamma}) \\
& \leq C t^\delta,
\end{align*}
where the last inequality follows from \eqref{eq:chernoff}. By the Borel--Cantelli lemma, the result then holds for all $t$ of the form $t=2^{-N}$ for integer $N$. One can then deduce the result for general $t$ by standard comparisons with the dyadic case.
\qed

For property (A4), we just treat the first supremum: the second one can be handled the same way.  We first get rid of the $\ln_+$ with the following lemma

\begin{lemma}\label{ridlog}
For each $\epsilon>0$, there exists a random constant 
$C_\epsilon=C(\epsilon,X)$ such that $\Pb^X$-a.s., for all $x\in \T$ and $r<1$
\begin{equation}\label{modmu}
\int_{B(x,r)}\ln_+\frac{1}{d_\T(x,z)}M_\gamma(dz)\leq C_\epsilon M_\gamma(B(x,r))^{1-\epsilon}.
\end{equation}
\end{lemma}

\proof Let us set $\bar{\mu}(x,r)= \int_{B(x,r)}  \ln\frac{1}{|x-z|} \,M_\gamma(dz)$. Also recall  \eqref{holderM} and set $\delta=(\alpha-\epsilon)\epsilon$.
Then we have
\begin{align*}
\bar{\mu}(x,r)=&\sum_{n\geq \frac{-\ln r}{\ln 2}}\int_{2^{-n}\leq |z-x|\leq 2^{-(n-1)}}  \ln\frac{1}{|x-z|} \,M_\gamma(dz)\\
\leq &\sum_{n\geq \frac{-\ln r}{\ln 2}}   n\ln2 \,M_\gamma(B(x,2^{-(n-1)}))\\
\leq & \sum_{n\geq \frac{-\ln r}{\ln 2}}   n\ln2 \,M_\gamma(B(x,2^{-(n-1)}))^{1-\epsilon} M_\gamma(B(x,2^{-(n-1)}))^\epsilon\\
\leq &M_\gamma(B(x,r))^{1-\epsilon}C^\epsilon\sum_{n\geq \frac{-\ln r}{\ln 2}}   n\ln2 \,  2^{-\delta(n-1)}.
\end{align*}
The latter series converges and can be bounded independently of $x,r$.\qed

By Lemma \ref{ridlog} we now get for all $\epsilon>0$, and all $y\in\T$ such that $|y|\leq t$,
\begin{align}
\int_{B(y,t)}\ln_+\frac{t}{d_\T(x,z)}M_\gamma(dz) & \leq \int_{B(y,t)}\ln_+\frac{1}{d_\T(x,z)}M_\gamma(dz)\nonumber\\
&\leq C_\epsilon M_\gamma(B(y,r))^{1-\epsilon}\nonumber\\
&\leq C_\epsilon M_\gamma(B(0,2r))^{1-\epsilon}.\label{RR}
\end{align}
Finally, we observe that for all $\eta' >0$, we have from \eqref{eq:chernoff},
\begin{align*}
 \Pb^X ( M_\gamma(B(0,t))   \geq t^{(2+\gamma^2/2-\eta')}  )   \leq C t^{\eta''},
\end{align*}
for some $\eta'' > 0$. Then by the Borel-Cantelli lemma there exists some random constant $C>0$ such that for all dyadic $t$ (i.e. $t$ of the form $2^{-N}$) one has
\begin{equation*}
M_\gamma(B(0,t))   \leq C t^{(2+\gamma^2/2-\eta')}.
\end{equation*} 
One can then reinforce the above inequality to all $t$ by standard comparisons. By combining with \eqref{RR} and since $\epsilon,\eta'$ can be chosen as small as we want, we get property (A4).

\subsubsection*{Step 2: Reduction to drifted Brownian motion in a thin tube}
From now on, we assume that $t\in(0,1)$ is fixed. The symbol $c$ will denote a positive numerical constant whose value may change from line to line. It may sometimes depend on other (deterministic) constants, if mentioned. Let $\Bvec = (B^1,B^2)$ be a standard Brownian motion on the torus $\T$, i.e.\ $B^1$ and $B^2$ are two independent Brownian motions on the circle $[-1,2]/_{-1\sim 2}$. 
Recall the definition of the additive functional 
\[
F(t) = F_{\Bvec}(t) = \int_0^t e^{\gamma X(\Bvec_s)-\frac{\gamma^2}{2}\E^X[X^2]}\,ds.
\]

We will want to force the second coordinate to stay in the interval $[-t,t]$ during the whole time. To formalize this, denote by $p_s^\circ(x,y)$ the transition density of the (one-dimensional) Brownian motion killed upon exiting the interval $[-t,t]$, i.e. \cite[p.~342]{Feller} for $x,y\in [-t,t]$
\begin{equation}
 \label{eq:pcirc}
 p_s^\circ(x,y) = \frac 1 t \sum_{k=1}^\infty e^{-\frac{k\pi^2}{8t^2}s}\sin\left(\frac{\pi k (x+t)}{2t}\right)\sin\left(\frac{\pi k (y+t)}{2t}\right).
\end{equation}
Here, and in the sequel, we write $A\asymp_\rho B$ if $\rho^{-1}B \leq A \leq \rho B$. The following lemma is standard:
\begin{lemma}\label{BMstrip} 1) For every $\rho > 1$ there exists a numerical constant $c=c(\rho)>0$, such that for all $s\geq ct^2$ and  $x,y\in [-t,t]$
\begin{equation}
 \label{eq:pcirc_estimate}
 p_s^\circ(x,y) \asymp_\rho \frac 1 t e^{-\frac{\pi^2}{8t^2}s}\cos\left(\frac{\pi x}{2t}\right)\cos\left(\frac{\pi y}{2t}\right).
\end{equation}
2) For every $c>0$ there exists $\rho = \rho(c)$, such that \eqref{eq:pcirc_estimate} holds for all $s\geq ct^2$.
\end{lemma}

\noindent {\it Proof of Lemma \ref{BMstrip}.} 1) use the inequality $\sin kx \leq k\sin x$ for $x\in[0,\pi]$. \\ 2) use \eqref{eq:pcirc} and a representation of $p_s^\circ$ in terms of Gaussian kernels \cite[p.~341]{Feller}, effective for small $s$.\qed

Let $B^*$ be a Brownian motion conditioned to stay forever in the interval $[-t,t]$, formally this is the Doob-transform of Brownian motion killed at $-t$ and $t$ with respect to the space-time harmonic function $h(x,s) = \cos(\pi x/2t)e^{\pi^2s/(8t^2)}$ \cite{Knight}. Its transition density is given for $x,y\in [-t,t]$ by 
\begin{equation}
 \label{eq:pstar}
 p_s^*(x,y) =  \frac{\cos\left(\frac{\pi y}{2t}\right)}{\cos\left(\frac{\pi x}{2t}\right)} e^{\frac{\pi^2}{8t^2}s}p_s^\circ(x,y) \asymp_\rho \frac{1}{t} \cos^2\left(\frac{\pi x}{2t}\right), 
\end{equation}
where the last inequality holds either for every $\rho>1$ and $s\geq c(\rho)t^2$ or for every $s\geq ct^2$ with $c>0$ and $\rho=\rho(c)$.

For a one-dimensional diffusion $B$, we denote by $\Loct{B}{x}{s}$ its local time at the point $x$ and time $s$. We will later need the fact that for all $y\in[-t/2,t/s]$ and all $x$,
\begin{equation}
\label{eq:Green_star}
\Emp{B^*}{0}[\Loct{B^*}{x}{s}] \leq c\left(\frac{s}{t}+t\right)\quad\forall s\ge0.
\end{equation}
The proof of this relation  follows easily from \eqref{eq:pstar}: split the interval $[0,s]$ into two parts $[0,ct^2]$ and $[ct^2,s]$. Then use  \eqref{eq:pstar} on $[ct^2,s]$ and estimate the integral of the transition density  of Brownian motion killed upon exiting $[-t,t]$ on $[0,ct^2]$ with the help of \eqref{eq:pcirc} to show that it is bounded by $ct$.

We now define the process $\Bvec^* = (B^1,B^*)$. We can relate it to $\Bvec$ through the following lemma:
\begin{lemma}
\label{lem:Bstar}
There exists a numerical constant $c > 0$, such that for $t\leq c^{-1}$, and any $\bx=(\xi_1,\xi_2)$ and $\by=(y_1,y_2)$ with
$|\xi_2|,|y_2|\leq t/2$, we have
\[
  \Ebr \Bvec \bx \by   t [e^{-\lambda F_{\Bvec}(t)}] \geq e^{-c/t} 
  \Ebr {\Bvec^*} \bx \by t [e^{-\lambda F_{\Bvec^*}(t)}]
\]
\end{lemma}
\begin{proof}
 Since the two coordinates of the processes $\Bvec$ and $\Bvec^*$ are independent, it suffices to show that for every non-negative bounded functional $f$, we have (with $B$ a one-dimensional Brownian motion),
 \begin{equation}
  \label{eq:foij}
  \Ebr B 
  {\xi_2} {y_2}
  t[f(B_s;s\leq t)] \geq e^{-c/t} \Ebr {B^*} 
  {\xi_2} {y_2}
   t[f(B^*_s;s\leq t)].
 \end{equation}
Denote by $B^\circ$ a Brownian motion killed upon exiting $[-t,t]$ . Since $B^*$ is a Doob transform of $B^\circ$, the bridges associated to both processes have the same law. Furthermore, we have for every non-negative measurable function $g$, with $p_s(x,y) = (2\pi s)^{-1/2} e^{-(x-y)^2/2t}$,
\begin{align*}
  \int &\Ebr B {\xi_2} y t [f(B_s;s\leq t)\ind_{\{|B_s|\leq t,\,\forall s\leq t\}}]p_t({\xi_2},y)g(y)\,dy \\
  &= \Emp B {\xi_2} [f(B_s;s\leq t)\ind_{\{|B_s|\leq t,\,\forall s\leq t\}}g(B_t)]\\
  &= \Emp {B^\circ} {\xi_2} [f(B^\circ_s;s\leq t)g(B^\circ_t)]\\
  &= \int \Ebr {B^\circ} {\xi_2} y t [f(B^\circ_s;s\leq t)]p^\circ_t({\xi_2},y)g(y)\,dy,
\end{align*}
such that, in particular,
\begin{align*}
  \Ebr B {\xi_2} {y_2} t [f(B_s;s\leq t)\Ind{|B_s|\leq t\,\forall s\leq t}] = \Ebr {B^\circ} {\xi_2} {y_2} t [f(B^\circ_s;s\leq t)]\frac{p^\circ_t(\xi_2,y_2)}{p_t(\xi_2,y_2)}.
\end{align*}
By \eqref{eq:pcirc_estimate}, we now have ${p^\circ_t(\xi_2,y_2)}/{p_t(\xi_2,y_2)} \geq e^{-c/t}$ for $t\leq c^{-1}$. Together with the previous calculations, this gives,
\begin{align*}
  \Ebr B {\xi_2} {y_2} t [f(B_s;s\leq t)] &\geq \Ebr B {\xi_2} {y_2} t [f(B_s;s\leq t)\Ind{|B_s|\leq t\,\forall s\leq t}]\\
  &\geq e^{-c/t}\Ebr {B^\circ} {\xi_2} {y_2} t [f(B^\circ_s;s\leq t)] \\
  &= e^{-c/t} \Ebr {B^*} {\xi_2} {y_2} t [f(B^*_s;s\leq t)],
\end{align*}
which is \eqref{eq:foij}.
\end{proof}

The next step is to first prove a result similar to Theorem~\ref{th:lower_bound} for the process $\Bvec^*$ with a suitable drift. Define the two-dimensional processes 
\[
\Bvec^\pm = (\Bvec^\pm_s;s\ge0) = ((B^1_s\pm s/t,B^*_s);s\ge0).
\]
We then have
\begin{proposition}
\label{prop:Bpm}
Suppose the assumptions (A1-5) hold and let $\eps>0$ be the same one as in (A3).
Then there exist constants $c = c(\eps)$ and $t_0 = t_0(\eps)>0$, such that,  for all $t\leq T_0\wedge t_0$, for all measurable $A\subset [1/2-\sqrt t,1/2+\sqrt t]\times[-t,t]$ such that $\Leb(A) \geq \eps t^{3/2}$, 
\begin{align*}
  &\inf_{\bx\in B_{t/2}(\0)}
  \Emp {\Bvec^+}{\bx}[e^{-\lambda F_{\Bvec^+}(t/2)}\Ind{\Bvec^+_{t/2} \in A}] \geq e^{-c  (\lambda t^{1+\gamma^2/4-\eta} +t^{-1})}\quad\text{and}\\
  &\inf_{\by\in B_{t/2}(\1)}
 \Emp {\Bvec^-}{\by}[e^{-\lambda F_{\Bvec^-}(t/2)}\Ind{\Bvec^-_{t/2} \in A}] \geq e^{-c (\lambda t^{1+\gamma^2/4-\eta} +t^{-1})}.
\end{align*}
\end{proposition}

We furthermore recall a probably standard lemma.
\begin{lemma}
 \label{lem:split}
 Let $f,g$ be non-negative measurable functions on a probability space $(\Omega,\mathcal B,\nu)$. Then,
 \[
  \int f(x)g(x)\,\nu(dx) \geq \left(\inf_{A\in\mathcal B,\,\nu(A)\geq 1/2} \int_A f(x)\,\nu(dx)\right)\left(\inf_{A\in\mathcal B,\,\nu(A)\geq 1/2} \int_A g(x)\,\nu(dx)\right). 
 \]
\end{lemma}
\begin{proof}
 Fix $A\in\mathcal B$ and write 
 \[
  a := \inf_{A\in\mathcal B,\,\nu(A)\geq 1/2} \int_A f(x)\,\nu(dx),\quad b := \inf_{A\in\mathcal B,\,\nu(A)\geq 1/2} \int_A g(x)\,\nu(dx).
 \]
Define the set $A_g = \{x\in\R^2: g(x) \geq b \}$. We claim that $\nu(A_g)\geq 1/2$. To prove this, suppose that $\nu(A_g) <1/2$. Then its complement satisfies $\nu(A_g^c) \geq 1/2$ and therefore, by the definition of $b$,
\[
\int_{A_g^c} g(x)\,\nu(dx) \geq b.
\]
But on the other hand, since $g(x) < b$ on $A_g^c$ and $\nu(A_g^c) \geq 1/2 > 0$, 
\[
\int_{A_g^c} g(x)\,\nu(dx) < b\nu(A_g^c) \leq b,
\]
which is a contradiction to the previous equation. It therefore follows that $\nu(A_g) \geq 1/2$. This now gives
\[
\int f(x)g(x)\,\nu(dx) \geq \int_{A_g} f(x)g(x)\,\nu(dx) \geq b \int_{A_g} f(x)\,\nu(dx) \geq ab,
\]
which proves the statement.
\end{proof}

We next show how Proposition~\ref{prop:Bpm} and Lemma~\ref{lem:Bstar} 
imply Theorem~\ref{th:lower_bound}.

\begin{proof}[Proof of Theorem~\ref{th:lower_bound} assuming  
    Proposition~\ref{prop:Bpm} and Lemma~\ref{lem:Bstar}] 
Throughout the proof, $c$ and $t_0$ denote some numerical constants 
whose value may change from line to line.
Let $\eta>0$ be arbitrary but fixed and let $\eps = \eps_0$, where $\eps_0>0$ is some numerical constant to be defined later. By the previous step, assumptions (A1-5) then hold for some random variables $W_0,\ldots,W_n\geq 0$ and $T_0>0$. By Lemma~\ref{lem:Bstar}, it is enough to show that for $t\leq T_0\wedge t_0$
and $\bx\in B_{t/2}(\0),
\by\in B_{t/2}(\1)$,
\begin{equation}
\label{eq:Bstar}
\Ebr {\Bvec^*} \bx \by t [e^{-\lambda F_{\Bvec^*}(t)}] \geq e^{-c  (\lambda t^{1+\gamma^2/4-\eta} +t^{-1})}.
\end{equation}
Fix $\bx=(\xi_1,\xi_2),\by=(y_1,y_2)$ as above; 
all the estimates in the sequel will be uniform in this choice.
Let $\mubr$ be the law of $\Bvec^*(t/2)$ 
under $\Pbr {\Bvec^*} \bx \by t$ and define $f(\bw) = \Ebr {\Bvec^*} \bx \bw {t/2} [e^{-\lambda F_{\Bvec^*}(t/2)}]$ and $g(\bw) = \Ebr {\Bvec^*} \bw \by {t/2} [e^{-\lambda F_{\Bvec^*}(t/2)}]$. Note that the functions $f(\cdot)$ and $g(\cdot)$ are continuous and hence measurable (the argument is the same as that   in \cite[Theorem 3.3]{spectral}). 
We now claim that for $t\leq T_0\wedge t_0$,
\begin{equation}
 \label{eq:f}
 \inf_{A\subset\R^2,\,\mubr(A)\geq 1/2} \int_A f(\bw)\,\mubr(d\bw) \geq e^{-c  (\lambda t^{1+\gamma^2/4-\eta} +t^{-1})}.
\end{equation}
We want to prove this using Proposition~\ref{prop:Bpm}. Let $A\subset\R^2$ be measurable and such that $\mubr(A)\geq 1/2$. We first verify that the assumption on $A$ in Proposition~\ref{prop:Bpm} is verified. For this, we note that the measure $\mubr$ has the following explicit form 
\begin{equation*}
\mubr(dx,dy) = \frac{p_{t/2}(\xi_1,x)p_{t/2}(x,y_1)}{p_t(\xi_1,y_1)}\times
\frac{p^\circ_{t/2}(\xi_2,y)p^\circ_{t/2}(y,y_2)}{p^\circ_t(\xi_2,y_2)}\,dx\,dy.
\end{equation*}
(Note again that because $B^*$ is a space-time Doob transform of $B^\circ$, the right side of the last expression would be the same if $p^\circ$ is replaced by 
$p^*$.)
By \eqref{eq:pstar}, there exists then $t_0 > 0$, such that for all $t\leq t_0$,
\begin{align}
\label{eq:mubr}
\mubr(dx,dy) \asymp_2 \sqrt{\frac 2 {\pi t}} e^{-\frac 2 t \left(x-\frac 1 2\right)^2} \times \frac 1 t \cos^2\left(\frac{\pi y}{2t}\right)\,dx\,dy,
\end{align}
In particular, by well-known estimates on the Gaussian integral, we get
\[
\mubr([1/2-\sqrt t,1/2+\sqrt t]\times[-t,t]) \geq 3/4,
\]
such that with $A' = A\cap [1/2-\sqrt t,1/2+\sqrt t]\times[-t,t]$, we have $\mubr(A') \geq \mubr(A) - 1/4 \geq 1/4$. But by \eqref{eq:mubr}, we have for all $t\leq T_0\wedge t_0$, for some numerical constant $\eps_0$,
\begin{equation}
\label{eq:A'}
\Leb(A') \geq \eps_0 t^{3/2}.
\end{equation}

Now denote by $\mu$ the law of $\Bvec^+(t/2)$ under $\Pmp {\Bvec^+} 0$ (which is also the law of $\Bvec^-(t/2)$ under $\Pmp {\Bvec^-} 1$ by symmetry). Note that the bridges of the processes $\Bvec^+$ and $\Bvec^*$ are the same, 
in particular, $f(\bw) = \Ebr {\Bvec^+} \bx \bw {t/2} [e^{-\lambda F_{\Bvec^+}(t/2)}]$.
Proposition~\ref{prop:Bpm} (with $\eps = \eps_0$) now  gives
\begin{equation}
\label{eq:int_f_mu}
\int_{A'} f(\bw)\,\mu(d\bw) = \Emp {\Bvec^+}{\bx}[e^{-\lambda F_{\Bvec^+}(t/2)}\Ind{\Bvec^+_{t/2} \in A'}] \geq e^{-c  (\lambda t^{1+\gamma^2/4-\eta} +t^{-1})},
\end{equation}
It now suffices to show that the Radon--Nikodym derivative $d\mubr/d\mu \geq c^{-1}$ on $[1/2-\sqrt t,1/2+\sqrt t]\times[-t,t]$. For this, we note that $\mu$ satisfies by \eqref{eq:pstar}, for $t\leq t_0$,
\[
\mu(dx,dy) = p_{t/2}(\xi_1,x-1/2)\times p^*_{t/2}(\xi_2,y)\,dx\,dy \asymp_2 \frac{1}{\sqrt{\pi t}}e^{-\frac{(x-\frac 1 2)^2}{t}}\times \frac 1 t \cos^2\left(\frac{\pi y}{2t}\right)\,dx\,dy.
\]
Together with \eqref{eq:mubr}, this yields $d\mubr/d\mu \geq c^{-1}$ on $[1/2-\sqrt t,1/2+\sqrt t]\times[-t,t]$. With \eqref{eq:int_f_mu}, this now gives for $t\leq T_0\wedge t_0$,
\[
\int_A f(\bw)\,\mubr(d\bw) \geq \int_{A'} f(\bw)\,\mubr(d\bw) \geq c^{-1} 
\int_{A'} f(\bw)\,\mu(d\bw) \geq e^{-c  (\lambda t^{1+\gamma^2/4-\eta} +t^{-1})}.
\]
Since $A$ was chosen arbitrarily, this finally yields \eqref{eq:f}.

To finish the proof, we note that an equation analogous to \eqref{eq:f} holds for the function~$g$. With Lemma~\ref{lem:split}, this now yields,
\[
\Ebr {\Bvec^*} \bx \by t [e^{-\lambda F_{\Bvec^*}(t)}] = \int f(\bw)g(\bw)\,\mubr(d\bw) \geq e^{-2c  (\lambda t^{1+\gamma^2/4-\eta} +t^{-1})}.
\]
This proves \eqref{eq:Bstar} and therefore finishes the proof of the theorem.
\end{proof}

\subsubsection*{Step 3. The speeding-up strategy, proof of Proposition~\ref{prop:Bpm}}

We now proceed to the proof of Proposition~\ref{prop:Bpm}. Throughout, we fix
$\bx = (\xi_1,\xi_2)$ and $\by = (y_1,y_2)$ as in the statement of the proposition.
The key to the proof is to change the measure of the Brownian motion by penalizing it as soon as it enters regions where the multiplicative chaos measure $M_\gamma$ is large. This yields a new process for which the functional $F(t)$ typically is of the order of $t^{1+\gamma^2/4-\eta}$. This allows to bound from below the expectations in Proposition~\ref{prop:Bpm} using Jensen's formula.

We assume that the assumptions of Proposition~\ref{prop:Bpm} are verified for some fixed $\eta,\eps > 0$, for some random variables $W_0,\ldots,W_n\geq 0$ and $T_0>0$ and for some fixed $t\leq T_0$. The symbols $c$ and $t_0$ will denote positive numerical constants whose value may change from line to line and   may depend on $\eps$ and, if mentioned, on other (deterministic) constants.
We first define the piecewise constant function $\phi(x)$ by 
\begin{equation}
 \label{eq:phi}
\phi(x) =  \max(W_{k-1},W_k,W_{k+1}),\quad x\in[x_k-t,x_k+t],\ k=0,\ldots,n
\end{equation}
(with the convention that $W_{-1}=W_{n+1}=0$). 
We denote its primitive by $\Phi(x) = \int_0^x\phi(y)\,dy$. 
By assumption (A1), 
\begin{equation}
 \label{eq:Phi}
 \Phi(1) \leq 6/t.
\end{equation}
Then define for a process $\Bvec=(B^1,B^2)$ the additive functional $\mathcal A_{\Bvec}$ by 
\[
\mathcal A_{\Bvec}(s) = \frac 1 2 \int_0^s \left(\frac 1 t + \phi(B^1_s)\right)^2 - \left(\frac 1 t\right)^2\,ds.
\]

Define the process $\Gammavec$ whose first coordinate $\Gamma^1$ is a Brownian motion with drift $1/t + \phi(\cdot)$, i.e. a solution to the SDE
\(
 d\Gamma^1_s = (1/t + \phi(\Gamma^1_s))ds + dB_s,
\)
for a Brownian motion $B$, and the second coordinate $\Gamma^2$ independently performs Brownian motion conditioned on staying inside $[-t,t]$.

Now, let $\eps$ and $A$ be as in the statement of Proposition~\ref{prop:Bpm}. We will only prove the first inequality in the statement of Proposition~\ref{prop:Bpm}, the proof of the other inequality is similar. Define the sets
\begin{align*}
\mathcal I &= \{k\in\{0,\ldots,n\}:S_k\cap[1/2-\sqrt t,1/2+\sqrt t]\times [-t,t]\ne \emptyset\}\\
\mathcal J &= \{k\in\mathcal I: \phi(x_k) \ne 0\text{ or } M_\gamma(S_k) > t^{2+\gamma^2/4-\eta}\}.
\end{align*}
By assumption (A3) and the definition of $\phi$, we have
\[
\Leb\left(\bigcup_{k\in\mathcal J} S_k\right) = t^2\#\mathcal J \leq (\eps/2)t^{3/2},
\]
whence, 
\[
\Leb\left(A \cap \bigcup_{k\in\mathcal I\backslash\mathcal J} S_k\right) \geq \Leb(A) - \Leb\left(\bigcup_{k\in\mathcal J} S_k\right) \geq (\eps/2)t^{3/2}.
\]
In particular, since $\#\mathcal I \leq 2/\sqrt t + 2$, there exists $K\in\mathcal I\backslash\mathcal J$, such that with  $t\leq T_0\wedge t_0$,
\begin{equation}
\label{eq:leb_A_SK}
\Leb(A\cap S_K) \geq (\eps/10) t^2.
\end{equation}
We now define a process $\Deltavec=(\Delta^1,\Delta^2)$ as follows: it is equal to the process $\Gammavec$ until its first coordinate equals $x_K$, after which the first coordinate moves according to the process $x_K+B^*$ (recall that $B^*$ was defined above to be a Brownian motion without drift conditioned to stay forever in $[-t, t]$).
We now have the following lemma:
\begin{lemma}
 \label{lem:Bvec_Delta} With $\eps$, $A$ and $K$ as above, we have for all $t\leq T_0\wedge t_0$,
 \begin{multline*}
  \Emp {\Bvec^+}{\bx}[e^{-\lambda F_{\Bvec^+}(t/2)}\Ind{\Bvec^+_{t/2} \in A}] \geq\\
  \Emp {\Deltavec}{\bx}[\exp\left(-\lambda F_{\Deltavec}(t/2) + \frac 1 2 \int \Loct{\Delta^1}{x}{t/2}\,d\phi(x)\right) \Ind{\Deltavec_{t/2} \in A}]\exp\left(-\frac c t\right).
 \end{multline*}
 Furthermore, for some numerical constant $c' = c'(\eps)> 0$,
 \[
  \Pmp {\Deltavec}{\bx}[\Deltavec_{t/2} \in A,\,\inf_{s\ge0}{\Delta^1_s} > -t] \geq 2c',
 \]
\end{lemma}

\begin{lemma}
 \label{lem:Delta} Let $\vartheta > 0$. Then, for some constant $c = c(\vartheta)$, for all $t\leq T_0\wedge t_0$,
 \begin{multline*}
  \Emp {\Deltavec}{\bx}[\left(\lambda (F_{\Deltavec}(t/2) - F_{\Deltavec}(\vartheta t^2)) + \frac 1 2 \int \Loct{\Delta^1}{x}{t/2}\,d\phi(x)\right)\Ind{\inf_{s\ge0}{\Delta^1_s} > -t}] \\
        \leq c  (\lambda t^{1+\gamma^2/4-\eta} +t^{-1}).
 \end{multline*}
\end{lemma}

\begin{lemma}
 \label{lem:Delta_small} Let $a > 0$. Then there exists $\vartheta = \vartheta(a) > 0$, such that for all $t\leq T_0\wedge t_0$,
 \[
  \Pmp {\Deltavec}{\bx}[F_{\Deltavec}(\vartheta t^2) \leq \vartheta^{-1} t^{1+\gamma^2/4-\eta}] \geq 1-a.
 \]
\end{lemma}

\begin{proof}[Proof of Lemma~\ref{lem:Bvec_Delta}]
 By the positivity of the functional $\mathcal A$ and Girsanov's theorem, we have for every $\mathcal F_{t/2}$-measurable non-negative random variable $H$,
 \begin{align*}
   \Emp {\Bvec^+}{\bx}[H] \geq  \Emp {\Bvec^+}{\bx}[H e^{ - \mathcal A_{\Bvec^+}(t/2)}] = \Emp {\Bvec^*}{\bx}[H e^{\frac 1 t B^1_{t/2}-\frac 1 2 \int_0^{t/2} \left(\frac 1 t + \phi(B^1_s)\right)^2\,ds}]
 \end{align*}
Tanaka's formula gives $\int_0^{t/2} \phi(B^1_s)\,dB^1_s = \Phi(B^1_{t/2}) - \frac 1 2 \int \Loct{B^1}{x}{t/2}\,d\phi(x)$. Girsanov's formula now yields
\[
 \Emp {\Bvec^*}{\bx}[H e^{\frac 1 t B^1_{t/2}-\frac 1 2 \int_0^{t/2} \left(\frac 1 t + \phi(B^1_s)\right)^2\,ds}] = \Emp {\Gammavec}{\bx}[H e^{ -\Phi(\Gamma^1_{t/2}) + \frac 1 2 \int \Loct{\Gamma^1}{x}{t/2}\,d\phi(x)}].
\]
We want to express the right-hand side in terms of $\Deltavec$ by a suitable change of measures. For this, denote by $\tau_K$ the hitting time of $x_K$ by the first coordinate and set $\tau = \tau_K\wedge t/2$. Then, for every  $\mathcal F_{t/2}$-measurable non-negative random variable $\widetilde H$, we have (the relation below is trivial if you observe that $|\Gamma^1_{t/2}-\Gamma^1_{\tau}|\leq t$ on the set $\{\forall s\in[\tau_K,t/2]:\Gamma^1_s\in[x_K-t,x_K+t]\}$)
\begin{multline*}
\Emp {\Gammavec}{\bx}[\widetilde H] \geq e^{-c/t}\times\\
\Emp {\Gammavec}{\bx}[\widetilde He^{-t^{-1}(\Gamma^1_{t/2}-\Gamma^1_{\tau})+(t/2-\tau)(\frac 1 {2t^2}-\frac 1 {8t^2})}\cos\left(\tfrac{\pi (\Gamma^1_{t/2}-x_K)\Ind{\tau_K<t/2}}{2t}\right)\Ind{\forall s\in[\tau_K,t/2]:\Gamma^1_s\in[x_K-t,x_K+t]}].
\end{multline*}
Now note that $\phi(x) \equiv 0$ on $[x_K-t,x_K+t]$ because $\phi(x)=\phi(x_K)$ for such $x$ and
$K\not\in \mathcal{J}$ implying that
$\phi(x_k)=0$.
By Girsanov's theorem 
and the definition of $B^*$ as a Doob transform of the killed Brownian motion, the last inequality gives 
\[
\Emp {\Gammavec}{\bx}[\widetilde H] \geq \Emp {\Deltavec}{\bx}[\widetilde H]e^{-c/t}.
\]
The previous inequalities and the monotonicity of $\Phi$ now give
\[
 \Emp {\Bvec^+}{\bx}[H] \geq \Emp {\Deltavec}{\bx}[H e^{\frac 1 2 \int \Loct{\Delta^1}{x}{t/2}\,d\phi(x)}] e^{-\frac c t - \Phi(1)}.
\]
Together with \eqref{eq:Phi}, this directly implies the first statement.

As for the second statement, we first note that standard comparison theorems show that
\begin{align*}
\Pmp{\Deltavec}{\bx}[\tau_K < t/2-t^2,\,\inf_{s\ge0}{\Delta^1_s} > -t] &= \Pmp{\Gammavec}{\bx}[\tau_K < t/2-t^2,\,\inf_{s\ge0}{\Gamma^1_s} > -t]\\
&\geq \Pmp{\Bvec^+}{\bx}[\tau_K < t/2-t^2,\,\inf_{s\ge0}{B^{+1}_s} > -t].
\end{align*}
Since $|x_K-1/2|\leq \sqrt{t}$, one easily sees that there exists a numerical constant $c_1$, such that for small enough $t$,
\[
\Pmp{\Bvec^+}{\bx}[\tau_K < t/2-t^2,\,\inf_{s\ge0}{B^{+1}_s} > -t] \geq c_1.
\]
Now, recall that after the time $\tau_K$, both coordinates of the process $\Deltavec$ perform Brownian motion conditioned to stay in $[x_K-t,x_K+t]$ and $[-t,t]$, respectively. By \eqref{eq:pstar}, we then have on the event $\{{\tau_K < t/2-t^2}\}$,
\[
\Pmp{\Deltavec}{\bx}[\Deltavec(t/2)\in A\mid\mathcal F_{\tau_K}] \geq \frac c {t^2}\int_{S_K\cap A} \cos^2(\pi (x-x_K)/2t)\cos^2(\pi y/2t)\,dx\,dy.
\]
Equation \eqref{eq:leb_A_SK} then gives, for some numerical constant $c_0=c_0(\eps)$,
\[
\Pmp{\Deltavec}{\bx}[\Deltavec(t/2)\in A\mid\mathcal F_{\tau_K}]\Ind{\tau_K < t/2-t^2} \geq c_0.
\]
The previous inequalities now yield
\[
  \Pmp {\Deltavec}{\bx}[\Deltavec_{t/2} \in A,\,\inf_{s\ge0}{\Delta^1_s} > -t] \geq c_0 \times c_1,
\]
which proves the second statement with $c' = c_0c_1/2$.
\end{proof}

\begin{proof}[Proof of Lemma~\ref{lem:Delta}]
 By definition of the process $\Deltavec$, its two coordinates are independent and the second coordinate is a Brownian motion started at $\xi_2$ and conditioned to stay in the interval $[-t,t]$. For every $s\geq \vartheta t^2$, its law is dominated by $ct^{-1}\times \Leb(\cdot\cap (-t,t))$, by \eqref{eq:pstar}. Here, and throughout the proof, the constant $c$ may depend on $\vartheta$. It follows that
\begin{align*}
 \Emp {\Deltavec}{\bx}[\left(F_{\Deltavec}(t/2) - F_{\Deltavec}(\vartheta t^2)\right)\Ind{\inf_{s\ge0}{\Delta^1_s} > -t}] 
\leq \frac c t\Emp {\Deltavec}{\bx}[\int_{(-t,\infty)\times(-t,t)} \Loct{\Delta^1}{x}{t/2}\,M_\gamma(dx,dy)]
\end{align*}
Denote by $\tau_K$ the hitting time of $x_K$ by the process $\Delta^1$. By \eqref{eq:Green_star}, we have for every $x\in(x_K-t,x_K+t)$,
\[
 \Emp {\Deltavec}{\bx}[(\Loct{\Delta^1}{x}{t/2} - \Loct{\Delta^1}{x}{\tau_K})\Ind{\tau_K<t/2}] \leq c,
\]
which implies,
\begin{multline*}
\Emp {\Deltavec}{\bx}[\int_{[-t,\infty)\times(-t,t)} \Loct{\Delta^1}{x}{t/2}\,M_\gamma(dx,dy)]\\
 \leq \Emp {\Gammavec}{\bx}[\int_{(-t,x_K]\times(-t,t)} \Loct{\Gamma^1}{x}{\infty}\,M_\gamma(dx,dy)] + c{M_\gamma(S_K)}.
\end{multline*}
Note that $M_\gamma(S_K)\leq t^{2+\gamma^2/4-\eta}$ by definition of $K$. Denoting by $G(x) = \Emp {\Gammavec}{\bx}[\Loct{\Gamma^1}{x}{\infty}]$ the Green's function of the process $\Gamma^1$, this yields,
\begin{equation}
\label{eq:G1}
\Emp {\Deltavec}{\bx}[\left(F_{\Deltavec}(t/2) - F_{\Deltavec}(\vartheta t^2)\right)\Ind{\inf_{s\ge0}{\Delta^1_s} > -t}] \leq \frac c t \int_{(-t,x_K]\times(-t,t)}G(x)\,M_\gamma(dx,dy) + ct^{1+\gamma^2/4-\eta}.
\end{equation}
Similarly, since $\phi \equiv 0$ in $(x_K-t,x_K+t)$,
\begin{equation}
\label{eq:G2}
 \Emp {\Deltavec}{\bx}[\left(\frac 1 2 \int \Loct{\Delta^1}{x}{t/2}\,d\phi(x)\right)\Ind{\inf_{s\ge0}{\Delta^1_s} > -t}] \leq \frac 12 \int_{(-t,x_K]}G(x)\,d\phi(x).
\end{equation}

It remains to bound the Green's function $G(x)$. We recall that $G(x) = h_x(\infty) - h_x(\xi_1)$, where $h_x$ is formally the solution to the equation
\[
\frac 12 h_x''(y) + \left(\frac 1 t + \phi(y)\right)h_x'(y) = \delta(x-y),\quad \forall y \leq x:h_x(y) = 0,
\]
which can be written rigorously as 
\begin{equation*}
 \begin{cases}
 \frac 12 h_x''(y) + \left(\frac 1 t + \phi(y)\right)h_x'(y) = 0, & y > x\\
 h_x(y) = 0, & y \leq x\\
 h_x'(x+) = 2
 \end{cases}
\end{equation*}
The explicit solution is
\begin{equation*}
 h_x(y) = \int_x^{y\vee x} 2 \exp\left(-\int_x^z \frac 1 t + \phi(r)\,dr\right)\,dz.
\end{equation*}
Elementary calculations now give for every $x$,
\begin{equation}
\label{eq:G_est}
 G(x) \leq h_x(\infty) \leq 2\min\left(t , \frac 1 {\frac 1 t+\phi(x+)} + te^{-(\hat x-x)\phi(x)}\right),\quad\hat x = \sup\{y\geq x: \phi(y) \geq \phi(x+)\}.
\end{equation}
In particular, for every $k\in\{0,\ldots,n\}$ and $x\in [x_k-t,x_k+t]$, we have by the definition of $\phi$,
\(G(x) \leq 2/(\frac 1 t+W_k).\) 
Together with \eqref{eq:G1} and assumption (A2), this yields
\begin{equation} 
\label{eq:almost_there}
\Emp {\Deltavec}{\bx}[\left(F_{\Deltavec}(t/2) - F_{\Deltavec}(\vartheta t^2)\right)\Ind{\inf_{s\ge0}{\Delta^1_s} > -t}] \leq c t^{1+\gamma^2/4-\eta}.
\end{equation}
Furthermore, by \eqref{eq:G_est} and the fact that $\phi$ is a step function with step size $2t$, we have $G(x) \leq c/\phi(x+)$ if $x$ is such that $\phi(x+) > \phi(x-)$. This yields
\begin{equation}
 \label{eq:almost_there2}
 \int_{(-t,x_K]}G(x)\,d\phi(x) \leq c \#\{x\in(-t,x_K): \phi(x+) > \phi(x-)\} \leq c/t.
\end{equation}
Equations \eqref{eq:almost_there}, \eqref{eq:G2} and \eqref{eq:almost_there2} now yield the lemma.
\end{proof}

\begin{proof}[Proof of Lemma~\ref{lem:Delta_small}]
Recall the processes $\Bvec^*$ and $\Bvec^+$ defined by $\Bvec^* = (B,B^*)$ and $\Bvec^+(s) = (B(s)+s/t,B^*(s))$. 
By \eqref{eq:phi} and Assumption (A5), we have $\phi\equiv 0$ in $[-t,t]$, such that the process $\Deltavec$ has the same law as $\Bvec^+$ until the first exit time out of the box $S_0$. 
Denote by $H_{R}$ the first exit time out of the disc with radius $R$ around the origin. 
By Brownian scaling, there exists then $\vartheta>0$, such that 
\begin{equation}
\label{eq:aspdfoij}
  \Pmp {\Deltavec}{\bx}[H_{3t/4} \leq \vartheta t^2] = \Pmp {\Bvec^+}{\bx}[H_{3t/4} \leq \vartheta t^2] \leq \Pmp {\Bvec^*}{\bx}[H_{t(3/4-\vartheta)} \leq \vartheta t^2] < a/2.
\end{equation}
Furthermore, by Girsanov's transform and the definition of $B^*$ as a Doob transform of Brownian motion killed upon exiting $[-t,t]$, we have (with $\Bvec=(B^1,B^2)$ a two-dimensional Brownian motion),
\begin{align*}
 \Emp {\Deltavec}{\bx}[F_{\Deltavec}(\vartheta t^2)\Ind{H_{3t/4} \geq \vartheta t^2}] &=
 \Emp {\Bvec}{\bx}[F_{\Bvec}(\vartheta t^2)\Ind{H_{3t/4} \geq \vartheta t^2}e^{B^1(\vartheta t^2)/t+\vartheta t^2(-\frac{1}{2t^2}+\frac 1 {8t^2})}\cos\left(\tfrac{\pi B^2_{\vartheta t^2}}{2t}\right)] \\
 &\leq c \Emp {\Bvec}{\bx}[F_{\Bvec}(\vartheta t^2)\Ind{H_{3t/4}\geq \vartheta t^2}].
\end{align*}
Recall (see e.g.\ \cite[Section~7.4]{ItoMcKean}) that the Green's function $G_{3t/4}(\bz)$ of the stopped process $\Bvec^{H_{3t/4}}$ satisfies $G_{3t/4}(\bz) \leq c\log(t/|\bz - \bx|)$. This gives
\begin{align*}
\Emp {\Bvec}{\bx}[F_{\Bvec}(\vartheta t^2)\Ind{H_{3t/4} \geq \vartheta t^2}] &\leq 
\Emp {\Bvec}{\bx}[F_{\Bvec}(H_{3t/4})]\\
&\leq c\int_{S_0} \log_+(t/|\bz-\bx|)\,M_\gamma(d\bz)\\
&\leq ct^{1+\gamma^2/4-\eta},
\end{align*}
where the last inequality follows from Assumption (A4). The two previous equations and Markov's inequality now give for small enough $\vartheta$,
\begin{align*}
 \Pmp {\Deltavec}{\bx}[F_{\Deltavec}(\vartheta t^2) \geq \vartheta^{-1}t^{1+\gamma^2/4-\eta},\,H_{3t/4} \geq \vartheta t^2] \leq c\vartheta \leq a/2.
\end{align*}
Together with \eqref{eq:aspdfoij}, this gives
 \begin{multline*}
\Pmp {\Deltavec}{\bx}[F_{\Deltavec}(\vartheta t^2)  \geq \vartheta^{-1}t^{1+\gamma^2/4-\eta}] \leq  \Pmp {\Deltavec}{\bx}[H_{3t/4} \leq \vartheta t^2] \\
+  \Pmp {\Deltavec}{\bx}[F_{\Deltavec}(\vartheta t^2) \geq \vartheta^{-1}t^{1+\gamma^2/4-\eta},\,H_{3t/4} \geq \vartheta t^2] < \frac a 2 + \frac a 2 = a.
 \end{multline*}
 This finishes the proof of the lemma.
\end{proof}

\begin{proof}[Proof of Proposition~\ref{prop:Bpm}]
We will only prove the first inequality, the proof of the other inequality is similar. The inequality will follow from Lemmas~\ref{lem:Bvec_Delta}, \ref{lem:Delta} and \ref{lem:Delta_small} above. Let $\eps$ and $A$ be as in the statement of the proposition and let $\Deltavec$ and $K$ be as above. Define the ``good'' event
\[
 G = \{\Deltavec_{t/2} \in A,\ \inf_{s\ge0}{\Delta^1_s} > -t,\ F_{\Deltavec}(\vartheta t^2) \leq t^{1+\gamma^2/4-\eta}\}.
\]
By Lemmas~\ref{lem:Bvec_Delta} and \ref{lem:Delta_small}, there exist $\vartheta =\vartheta(\eps) > 0$ and $c_0=c_0(\eps) > 0$, such that for all $t\leq t_0$, $\Pmp{\Deltavec}{\bx}[G] \geq c_0$. Write
\[
 \Upsilon = \lambda (F_{\Deltavec}(t/2) - F_{\Deltavec}(\vartheta t^2)) + \frac 1 2 \int \Loct{\Delta^1}{x}{t/2}\,d\phi(x).
\]
We now have
\begin{align*}
  &\Emp {\Deltavec}{\bx}[e^{-\Upsilon - \lambda F_{\Deltavec}(\vartheta t^2)} \Ind{\Deltavec_{t/2} \in A}]\\
 &\geq \Emp {\Deltavec}{\bx}[e^{-\Upsilon} \Indev{G}]e^{-\lambda  t^{1+\gamma^2/4-\eta}} && \text{by definition of $G$}\\
 &\geq \Emp {\Deltavec}{\bx}[e^{-\Upsilon}\mid G]e^{-\lambda  t^{1+\gamma^2/4-\eta}}\Pmp{\Deltavec}{\bx}[G] \\
 &\geq e^{- \Emp {\Deltavec}{\bx}[\Upsilon\mid G]}e^{-\lambda  t^{1+\gamma^2/4-\eta}}\Pmp{\Deltavec}{\bx}[G] && \text{by Jensen's inequality}\\
 &\geq e^{- \Emp {\Deltavec}{\bx}[\Upsilon]/c}e^{-\lambda  t^{1+\gamma^2/4-\eta}}c_0 && \text{by $\Pmp{\Deltavec}{\bx}[G] \geq c_0$.}\\
 & \geq e^{-c(\lambda  t^{1+\gamma^2/4-\eta} + 1/t)} && \text{by Lemma~\ref{lem:Delta}}.
\end{align*}
Together with Lemma~\ref{lem:Bvec_Delta}, this finishes the proof of the proposition.
\end{proof}

\subsection{Lower bound on the heat kernel; Harnack inequality}
\label{sec:harnack_lower}

This theorem will follow from the resolvent bound in Theorem~\ref{lowerboundreso} once we have established suitable near-diagonal estimates on the heat kernel using the Harnack inequality together with on-diagonal estimates.

Observe that the Harnack inequality holds for the standard Brownian motion on the torus. Thus, applying \cite{kim}, we deduce
\begin{theorem}\label{harnack}
For any open set $D\subset \T$ and any compact set $K\subset D$, there exists a constant $C$ only depending on $D,K$ such that for any nonnegative function $h$ harmonic in $D$ with respect to $\LB$ 
$$\sup_{x\in K}h(x)\leq C_K\inf_{x\in K}h(x). $$
\end{theorem}

Let us set for any open subset $D$ of $\T$, 
\begin{equation}
\widetilde{E}(D) =\sup_{x\in D}\Emp{\Bvec}{x}[\tau_{D}].
\end{equation}

We can then follow the proof of \cite[Proposition 5.3]{grig2} to get
\begin{proposition}\label{harnack+}
There exists a constant $\theta>0$ such that for any open subset $D$ of $\T$, any measurable bounded function $f$ on $D$,  for any ball $B(x,r)\subset D$ and all $\rho\in ]0,r]$
$${\rm esup }_{B(x,\rho)}\,u -{\rm einf }_{B(x,\rho)}\,u \leq 2\widetilde{E}(B(x,r)) \,{\rm esup }_{B(x,r)}\,|f|+4\big(\frac{\rho}{r}\big)^\theta {\rm esup }_{B(x,r)}\,|u|  $$
where we have set $u=G^Df$ and $G^Df(x)=\Emp{\Bvec}{x}[\int_0^{\tau_D}f(\LB_r)\,dr] $.
\end{proposition}


\begin{theorem}
\label{th:near-diag}
For each $\delta>0$, $\Pb^X$-almost surely, there exists a constant $C_\delta>0$ such that
$$\p^\gamma_t(x,x)-\inf_{y\in B(x,\rho)}\p^\gamma_t(x,y)\leq C_\delta\Big(\frac{\rho}{t^{\frac{1}{\alpha-\delta}}}\Big)^{\frac{\alpha\theta}{\alpha+1}} \sup_{y\in B(x,r)}\p^\gamma_{t/2}(y,y)$$
with 
$$r=\big(t \rho^\theta\big)^{\frac{(\alpha-\delta)\theta}{\alpha-\delta+\theta}}.$$
\end{theorem}

\noindent {\it Proof.} It suffices to apply Proposition \ref{harnack+}  as explained in \cite[Lemma 5.10]{grig2} while taking care of using Lemma \ref{UB6} below to estimate $\widetilde{E}(B(x,r))$ instead of the estimate on the function $F$ used in \cite[Lemma 5.10]{grig2} .\qed

\begin{lemma}\label{UB6}
For each $\delta>0$, $\Pb^X$-almost surely, there is a random constant $C>0$ such that for all $x\in\T$ and $r\leq 1$
$$\widetilde{E}(B(x,r))\leq Cr^{\alpha-\delta}.$$
\end{lemma}

\noindent {\it Proof.} It suffices to notice that
$$\Emp{\Bvec}{y}[\tau_{B(x,r)}]=\int_{B(x,r)}G^R(y,z)\,M_\gamma(dz)\leq  \int_\T\ln\frac{10}{d_\T(x,z)}\,M_\gamma(dz),$$
and to apply Lemma \ref{UB5}.\qed

\begin{lemma}
\label{lem:LB-diag}
There is a constant $C>0$ such that
$$\inf_{x\in \T}\,M_\gamma(B(x,t^{\frac{1}{\beta_\delta}}))\,\p^\gamma_t(x,x)\geq C.$$
Consequently, $\Pb^X$-almost surely, for some random constant $C'$
$$\inf_{x\in \T} \p^\gamma_t(x,x)\geq C't^{-\frac{\alpha_\delta}{\beta_\delta}}.$$
\end{lemma}

\noindent {\it Proof.} By applying Lemma \ref{UB2}, we can find $r>0$ such that for all $t\leq  r^{\beta_\delta}$
$$\sup_{x\in\T,t\leq r^{\beta_\delta}}\Pmp{\Bvec}{x}(\tau_{B(x,r)}\leq t)\leq \frac{1}{2}.$$
Then for all $t\leq  r^{\beta_\delta}$ and $x\in\T$, we have  
\begin{align*}
\int_{B(x,r)}\p^\gamma_t(x,z)\,M_\gamma(dz)&= \Pmp{\Bvec}{x}(\LB_t\in B(x,r)) \geq  \Pmp{\Bvec}{x}(\tau_{B(x,r)}> t)\\
&\geq 1-\Pmp{\Bvec}{x}(\tau_{B(x,r)}\leq t)\\
&\geq \frac{1}{2}.
\end{align*}
Therefore
\begin{align*}
\p^\gamma_{2t}(x,x)=&\int_\T \p^\gamma_{t}(x,z)\p^\gamma_{t }(z,x)\,M_\gamma(dz)\\
\geq &\int_{B(x,r)} \p^\gamma_{t}(x,z)^2\,M_\gamma(dz)\\
\geq &\frac{1}{M_\gamma(B(x,r))}\Big(\int_{B(x,r)} \p^\gamma_{t }(x,z)\,M_\gamma(dz)\Big)^2\\
\geq & \frac{1}{4 M_\gamma(B(x,r))}.
\end{align*}
We conclude with the relation $M_\gamma(B(x,r))\leq C' r^{\alpha_\delta}$ for some random constant $C'>0$.\qed

\begin{corollary}
\label{cor:near-diag-LB}
There exists a constant $\beta = \beta(\gamma)$ and a random variable $T_0>0$, such that for all $t\leq T_0$,
\[
\inf_{x\in\T}\inf_{y\in B(x,t^\beta)} \p_t(x,y) \geq 1.
\]
\end{corollary}
\begin{proof}
Combine Theorem~\ref{th:near-diag} with Lemmas~\ref{UB1} and \ref{lem:LB-diag}.
\end{proof}

We now have developed all the tools to prove Theorem~\ref{th:LB-heat-kernel}.

\begin{proof}[Proof of Theorem~\ref{th:LB-heat-kernel}]
Fix $\eta>0$ and $x\ne y\in\T$. We will first show that it is enough to get a lower bound on $\p_t(x,B(y,t^\beta)) = \int_{B(y,t^\beta)}\p_t(x,z)\,M_\gamma(dz)$ for some $\beta > 0$ for small enough $t$. For, we have for every $s\leq t$ and $r>0$,
\begin{equation*}
\p_{2t}(x,y) \geq \int_{B_r(y)} \p_s(x,z)\,\p_{2t-s}(z,y)\,M_\gamma(dz) \geq \p_s(x,B_r(y)) \inf_{z\in B_r(y)} \p_{2t-s}(z,y).
\end{equation*}
Corollary~\ref{cor:near-diag-LB} now ensures the existence of a constant $\beta=\beta(\gamma)>0$ and a random variable $T_0>0$, such that for all $t\leq T_0$ and $s\leq t$,
\[
\inf_{z\in B(y,t^\beta)} \p_{2t-s}(z,y) \geq \inf_{z\in B_{(2t-s)^\beta}(y)} \p_{2t-s}(z,y)\geq 1.
\]
The last two equations now yield
\begin{equation}
\p_{2t}(x,y) \geq \p_s(x,B(y,t^\beta))\,\quad\text{for all $t\leq T_0$ and $s\leq t$}.
\label{eq:pt_ball}
\end{equation}

The definition of the resolvent $\r_\lambda$ now gives for every $\lambda \geq 0$,
\begin{align*}
\r_\lambda(x,B(y,t^\beta)) = \int_0^\infty e^{-\lambda s}\p_s(x,B(y,t^\beta))\,ds
\leq \int_0^t \p_s(x,B(y,t^\beta))\,ds + \lambda^{-1} e^{-\lambda t},
\end{align*}
where the last inequality follows from the fact that $\p_s(x,B(y,t^\beta)) \leq 1$ for every $s\ge0$, by definition. Together with \eqref{eq:pt_ball}, this yields for $t\leq T_0$ and $\lambda \geq 0$,
\begin{equation}
\p_{2t}(x,y) \geq t^{-1}\left(\r_\lambda(x,B(y,t^\beta)) - \lambda^{-1} e^{-\lambda t}\right).
\label{eq:p_r}
\end{equation}

 We now choose $\lambda$ in terms of $t$. Specifically, we set 
\[
\lambda_t = (2t)^{-\frac{2+\gamma^2/4-\eta}{1+\gamma^2/4-\eta}}, \quad t>0.
\]
By Theorem~\ref{lowerboundreso}, we now have for $t$ small enough (possibly adjusting the value of $\beta$),
\[
\r_{\lambda_t}(x,B(y,t^\beta)) \geq  \exp\Big(- {\lambda_t}^{\frac{1}{2+\gamma^2/4-\eta}}\Big)M_\gamma(B(y,t^\beta)).
\]
Together with \eqref{eq:p_r}, this now gives for $t$ small enough, with $\alpha := \gamma^2/4-\eta$, 
\begin{equation}
\label{eq:p2t}
\p_{2t}(x,y) \geq t^{-1}\left(M_\gamma(B(y,t^\beta))\exp\Big(- 2^{-\frac{1}{1+\alpha}}t^{-\frac{1}{1+\alpha}}\Big) - (2t)^{\frac{2+\alpha}{1+\alpha}} \exp\Big(-2^{-\frac{2+\alpha}{1+\alpha}}t^{-\frac{1}{1+\alpha}}\Big)\right).
\end{equation}
Assuming w.l.o.g.\ that $\eta < \gamma^2/4$, i.e. $\alpha > 0$, we have $2^{-\frac{1}{1+\alpha}} < 2^{-\frac{2+\alpha}{1+\alpha}}$. Furthermore, by  Lemma \ref{holderbelow} below, we have \(M_\gamma(B(y,t^\beta)) \geq t^{\beta'}\) for some constant $\beta'$ and $t$ small enough. Together with \eqref{eq:p2t}, this now yields the theorem for every $\eta' > \eta$. Since $\eta \in (0,\gamma^2/4)$ was chosen arbitrarily, this finishes the proof of the theorem.
\end{proof}

\begin{lemma}\label{holderbelow}
$\Pb^X$-almost surely, there exists a constant $C>0$ and $\beta'>0$ such that for all $x\in\T$ and $r<1$
$$M_\gamma(B(x,r))\geq C r^{\beta'}.$$
\end{lemma}

\noindent {\it Proof.} Recall \eqref{powerlaw}. Let us denote by $(S_k^n)_{k\leq 2^{2n}}$ a partition of the torus into $2^{2n}$ identical squares with side length $2^{-n}$. We have for any $\chi>0$ and any $p>0$
$$\mathbb{P}^X(\min_{k\leq 2^{2n}}M_\gamma(S_k^n)\leq 2^{-\chi n})\leq 2^{-\chi p n}\E^X[\sum_{k\leq 2^{2n}} M_\gamma(S_k^n)^{-p}]\leq C2^{-(\chi p-2+\xi(-p))n}.$$
Then it suffices to choose $p>0$ such that $\chi p-2+\xi(-p)>0$ and to use Borel-Cantelli.\qed

\subsection{Sampling the endpoints according to \texorpdfstring{$M_\gamma$, $\gamma^2 \leq 4/3$}{M\_gamma, gamma2 <= 4/3}}\label{subsampling1}
\label{sec:lower_bound_some_gamma}

 In this section, we show that the results of the previous sections are still valid when the endpoints are sampled according to the measure $M_\gamma$, at least in the case $\gamma^2 \leq 4/3$.

In order to be consistent with our previous analysis (i.e. short time behaviour of the heat kernel for endpoints $x,y$ with a fixed distance), we have to make sure that sampling these endpoints according to $M_\gamma$ will not produce points that are too close to each other, which may non-trivially perturb the regime $t<<|x-y|$ within which we work. So we will fix two arbitrary points, say $\0=(0,0)$ and $\1=(1,0)$, and sample the endpoints according to $M_\gamma$ in two small balls drawn around these two points.

So we fix $r<1/4$. We will sample the starting point in the ball $B(\0,r)$ and the final point in the ball $B(\1,r)$. For $x\in B(\0,r)$ and $y\in B(\1,r)$, we consider the segment linking $x$ to the point $y$.

We consider a fixed $t>0$ (small). We set $v=\frac{y-x}{|y-x|}$ and we consider a family $(x_k)_{0\leq k\leq n}$ such that 
\begin{equation}
n=\frac{|y-x|}{2t},\quad x_k=x+2tkv.
\end{equation}
Then we recover the segment $[x,y]$ with squares $(S_k)_{0\leq k\leq n}$ of side length $2t$, each square $S_k$ being centered at $x_k$.

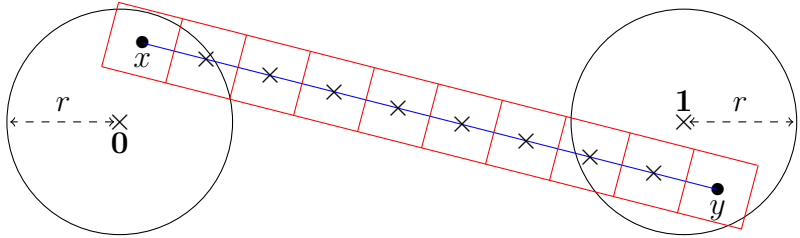
\begin{figure}[h] 
\begin{center}
\begin{tikzpicture}[scale=1.5] 
\draw (0,0) circle (1) ;
\draw (0.2,0.7) node {$\bullet$};
\draw (0.2,0.7) node[below] {$x$};
\draw[dashed,<->] (-0.05,0)--(-0.97,0);
\draw (-0.5,0) node[above] {$r$};
\draw (5,0) circle (1) ;
\draw[dashed,<->] (5.05,0)--(5.97,0);
\draw (5.5,0) node[above] {$r$};
\draw (5.3,-0.6) node {$\bullet$};
\draw (5.3,-0.6) node[below] {$y$};
\draw (0,0) node {$\times$};
\draw (0,0) node[below] {$\0$};
\draw (5,0) node {$\times$};
\draw (5,0) node[above] {$\1$};
\draw [red] (-0.01,1.06) -- (5.65,-0.385) ;
\draw [red] (-0.16,0.49) -- (5.51,-0.95) ;
\draw [blue] (0.2,0.7)--(5.3,-0.6);
\foreach \x in {0,...,10}
{\draw[red] (-0.01+\x*0.969*0.585,1.06-0.247*\x*0.585) -- (-0.16+\x*0.969*0.585,0.49-0.247*\x*0.585);
    };   
\foreach \x in {1,...,8}
{\draw (0.2+\x*0.969*0.585,0.7-0.247*\x*0.585) node {$\times$}; };    
\end{tikzpicture}
\end{center}
\caption{Construction of the squares recovering the segment from $x$ to $y$.}
\label{fig:couplig1}
\end{figure}
We define the sequence $W_k$ by the following formula for $k \not \in \lbrace 0,1,n-1,n \rbrace$   

\[
 W'_k = t^{-2-\gamma^2/8+\eta/2}\sqrt{M_\gamma(S_k)},\quad W_k = \begin{cases}
                                                           W'_k,\quad\text{ if }W'_k \geq 1/t \\
                                                           0,\quad\text{ otherwise.}
                                                          \end{cases}
\]
and $W_k=0$ for $k \in \lbrace 0,1,n-1,n \rbrace$. We will prove the following: $\Pb^X$-a.s., for $M_\gamma$ almost every $x\in B(\0,r)$ and $M_\gamma$ almost every $y\in B(\1,r)$, for each $\eta ,\eps>0$, there exists $T_0>0$, such that the following holds for $t\leq T_0$:
\begin{itemize}
        \item (A1) and (A2) for $\gamma^2 \leq 2$,
        \item (A4) for $\gamma^2 \leq 4/3$,
        \item (A3) and (A5) for all $\gamma$, i.e. $\gamma^2 <4$.
\end{itemize}

Note that the proof of Theorem~\ref{lowerboundreso} only depends on these assumptions on the field $X$, whence the Theorem still holds in the current case as long as $\gamma^2 \leq 4/3$. The results from Section~\ref{sec:harnack_lower} can then be applied to yield the analog of Theorem~\ref{th:lower_bound} as well.

We now turn to the proofs. Define the probability measure $\Q$ on $ \Omega\times B(\0,r)\times B(\1,r)$ by 
$$\Q(A)=c_r\E^X\Big[\int_{B(\0,r)\times B(\1,r)} \ind_A\, \, M_\gamma(dx)M_\gamma(dy)\Big]$$
where the constant $c_r$ is chosen so as to turn $\Q$ into a probability measure.

First, we will use the following lemma:

\begin{lemma}\label{lemmeMultiM}
Let $\eta'>0$ and set 
$$D_t(x,y)=\sum_{k=2}^{n-2} t^{-\gamma^2/8+\eta'}\sqrt{M_\gamma(S_k)}.$$
If $\gamma^2 \leq 2$, then there exists some random constant $C>0$ such that
\begin{equation*}
\Q \text{-a.s.},\forall t\leq 1,\quad \quad D_t(x,y)  \leq C.
\end{equation*}

\end{lemma}

\proof We have
\begin{align*}
  \E^\Q  \left [  D_t (x,y) \right  ]   
& \leq    t^{-\gamma^2/8+\eta'}\sum_{k=2}^{n-2}  \E^\Q  [   \sqrt{M_\gamma(S_k)}    ]  \\
&= t^{-\gamma^2/8+\eta'}\sum_{k=2}^{n-2} \int_{B(\0,r)\times B(\1,r)} \E^X  [ e^{\gamma X(x)+\gamma X(y)-\frac{\gamma^2}{2}\E[X(x)^2]-\frac{\gamma^2}{2}\E[X(y)^2]}  \sqrt{M_\gamma(S_k)}    ] dxdy
\end{align*}
Now let $K(x,y)$ denote the covariance kernel of the field $X$, i.e. the quantity in \eqref{covX}. This is bounded over $B(\0,r)\times B(\1,r)$. By Girsanov's transform,
\begin{align*}
  \E^\Q  \left [  D_t (x,y) \right  ]   
& =  t^{-\gamma^2/8+\eta'}\sum_{k=2}^{n-2}  \int_{B(\0,r)\times B(\1,r)}e^{\gamma^2 K(x,y)} \E^X \big [  \big(\int_{S_k}e^{\gamma^2K(x,z)+\gamma^2K(y,z)}M_\gamma(dz)\big)^{1/2}   \big ]dxdy \\
&\leq Ct^{-\gamma^2/8+\eta'}\sum_{k=2}^{n-2} \frac{1}{(tk)^{\frac{\gamma^2}{2}}} \E^X  [  (M_\gamma(S_k))^{1/2}   ]  \\
& \leq C t^{-\gamma^2/8+\eta'-\gamma^2/2+\zeta(1/2) }  \sum_{k=2}^{n-2}k^{-\frac{\gamma^2}{2}}
\end{align*}
where $C$ is some deterministic constant. For $\gamma^2\leq 2$, we now have $\sum_{k=2}^{n-2}k^{-\frac{\gamma^2}{2}} \leq c t^{\frac{\gamma^2}{2}-1-\eta'/2}$, such that
$$ \E^\Q  \left [  D_t (x,y) \right  ]  \leq C t^{-1-\gamma^2/8+\zeta(1/2)+\eta'/2 }  = C t^{\eta'/2 }.$$
We conclude as in the proof of Lemma \ref{lemmeMulti}.\qed

 \medskip

Lemma~\ref{lemmeMultiM} now implies (A1) and (A2) as explained in Step 1 of Section \ref{lowerboundsection} (just after Lemma \ref{lemmeMulti}).

Property (A3) is proven similarly as in Section~\ref{lowerboundsection}. The only difference is that we have to deal with the probability measure $\Q$ instead of $\Pb^X$, but this is easy as we deal with quantities far away from $B(\0,r)$ and $B(\1,r)$. Indeed, since the covariance kernel \eqref{covX} of the field $X$ is uniformly bounded on $B(\0,r)\times B(\1,r)$, we get for every $k\in\{n/4,\ldots,3n/4\}$ and $p\ge 0$,
\[
\E^\Q [M_\gamma(S_k)^p] \le C \E^X[M_\gamma(S_k)^p].
\]
Hence, a Chernoff bound similar to \eqref{eq:chernoff} holds under $\Q$ for the boxes $S_k$, $k\in\{n/4,\ldots,3n/4\}$. Lemma~\ref{lemmeA3} and, as a consequence, property (A3) then hold with $\Pb$ replaced by $\Q$.

For property (A4), we first note that, as before, by Lemma~\ref{ridlog} it is enough to show that for $\eta' = \eta/2$,
 $$\sup_{t\leq 1}t^{-1-\gamma^2/4+\eta'}\Big(\sup_{x'\in B(x,t)} M_\gamma(B(x',t)) + \sup_{y'\in B(y,t)} M_\gamma(B(y',t)) \Big)<+\infty.$$
Furthermore, because of the relation $B(x',t)\subset B(x,2t)$ for all $x'\in B(x,t)$ (recall that $t\leq 1$), this is reduced to 
 $$\sup_{t\leq 1}t^{-1-\gamma^2/4+\eta'}\Big(M_\gamma(B(x,t)) + M_\gamma(B(y,t)) \Big)<+\infty.$$
We now use the well-known fact \cite{cf:Kah} that for all $\eps>0$, $\Pb^X$-a.s., for $M_\gamma$-almost every $x\in\T$, 
\begin{equation}
\sup_{r\leq 1}r^{-2+\gamma^2/2+\eps}M_\gamma(B(x,r))<+\infty.
\label{eq:ball_Mgamma}
\end{equation}
Assumption (A4) is therefore verified if $-1-\gamma^2/4 \geq -2+\gamma^2/2$, or, $\gamma^2 \leq 4/3$, as claimed.

\subsection{Sampling the endpoints according to \texorpdfstring{$M_\gamma$, all $\gamma$}{M\_gamma, all gamma}}\label{subsampling2}
\label{sec:lower_bound_all_gamma}

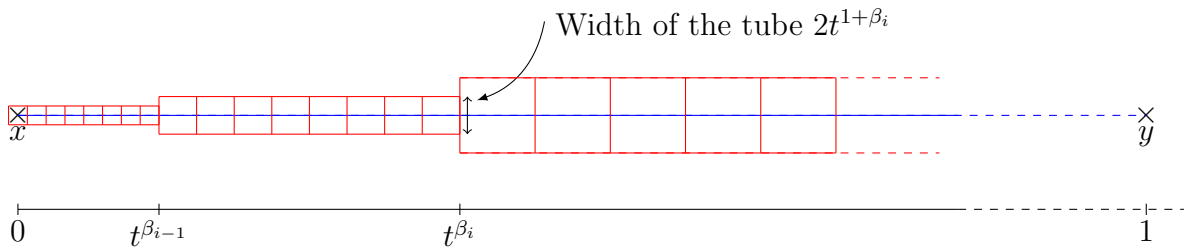
\begin{figure}[h] 
\begin{center}
\begin{tikzpicture}[scale=2.5] 
\draw (0,0) node {$\times$};
\draw (0,0) node[below] {$x$};
\draw (6,0) node {$\times$};
\draw (6,0) node[below] {$y$};
\draw [red] (-0.05,0.05) -- (0.75,0.05) ;
\draw [red] (-0.05,-0.05) -- (0.75,-0.05) ;
\draw [blue] (0,0)--(5,0);
\draw [blue,dashed] (0,0)--(6,0);
\draw [black,dashed] (5,-0.5)--(6.2,-0.5);
\draw (0,-0.5)--(5,-0.5);
\draw (0.75,-0.53)--(0.75,-0.47);
\draw (2.35,-0.53)--(2.35,-0.47);
\draw (0,-0.53)--(0,-0.47);
\draw (6,-0.53)--(6,-0.47);
\draw (0.75,-0.5) node[below] {$t^{\beta_{i-1}}$};
\draw (2.35,-0.5) node[below] {$t^{\beta_{i}}$};
\draw (0,-0.5) node[below] {$0$};
\draw (6,-0.5) node[below] {$1$};
\draw[<->] (2.39,-0.1) -- (2.39,0.1);
\draw[<-,>=latex] (2.44,0.08) to[bend right] (2.8,0.5) node[right] {Width of the tube $2t^{1+\beta_i}$};
\foreach \x in {0,...,8}
{\draw[red] (-0.05+\x*0.1,0.05) -- (-0.05+\x*0.1,-0.05);
   };    
\foreach \x in {0,...,8}
{\draw[red] (0.75+\x*0.2,0.1) -- (0.75+\x*0.2,-0.1);
   };    
   
\draw [red] (0.75,0.1) -- (2.35,0.1) ;
\draw [red] (0.75,-0.1) -- (2.35,-0.1) ;  
\foreach \x in {0,...,5}
{\draw[red] (2.35+\x*0.4,0.2) -- (2.35+\x*0.4,-0.2);
   };    
\draw [red] (2.35,0.2) -- (4.35,0.2) ;
\draw [red] (2.35,-0.2) -- (4.35,-0.2) ;  
\draw [red,dashed] (2.35,0.2) -- (4.9,0.2) ;
\draw [red,dashed] (2.35,-0.2) -- (4.9,-0.2) ;    
\end{tikzpicture}
\end{center}
\caption{Construction of the squares recovering the segment from $x$ to $y$.}
\label{fig:couplig2}
\end{figure}

Here, we explain how to obtain extend the lower bound from the previous section to all values of $\gamma$, with the same lower bound as in Theorem~\ref{th:lower_bound} for $\gamma^2\leq \frac{8}{3}$, but with a worse lower bound for higher values of $\gamma$. The idea is to apply the above strategy on several scales near the points $x$ and $y$. For simplicity, let us only consider the construction near $x$, the one near $y$ is similar by symmetry. As in Step 2 of the proof of Theorem~\ref{th:lower_bound}, we can then replace the Brownian bridge by a Brownian motion with drift $(y-x)/t$. 

We use a notation similar to Section~\ref{subsampling1}. We define the probability measure $\Q$ on $ \Omega\times B(\0,r)$ by 
$$\Q(A)=c_r\E^X\Big[\int_{B(\0,r)} \ind_A\, \, M_\gamma(dx)\Big]$$
where the constant $c_r$ is chosen so as to turn $\Q$ into a probability measure. We then adapt the following notational convention regarding the point $x$: when we work under $\Pb^X$, it is simply some fixed point $x\in\T$, when we work under $\Q$, it is the random point in $B(\0,r)$. As in the proof of Lemma~\ref{lemmeMultiM}, we have by Girsanov's transform, for any measurable functional $F$,
\begin{equation}
\label{eq:Q}
 \E^\Q\left[ F(X) \mid x\right] = c_r \E^X\left[F(X+\gamma^2 K(x,\cdot))\right],\quad \text{$\Q$-almost surely},
\end{equation}
where $K(x,y)$ denotes again the covariance kernel of the field $X$, i.e. the quantity in \eqref{covX}.

We now choose a sequence $\beta_1 > \cdots > \beta_N>0$, depending on $\gamma$, with $\beta_{i-1} -\beta_i \leq 1$ for all $i\geq 2$, and consider a cascade of tubes around the straight line between $x$ and $y$ of the form\footnote{We rotate and translate our coordinate system such that $x$ is the origin and $y$ lies on the horizontal axis.} $([t^{\beta_{i-1}},t^{\beta_i}] \times [-t^{1+\beta_i},t^{1+\beta_i}]  )_{2 \leq i \leq N}$, the first tube being $[-t^{\beta_0},t^{\beta_1}] \times [-t^{1+\beta_1},t^{1+\beta_1}]$ with $\beta_0 = 1+\beta_1$. We force the Brownian motion to cross each tube in the horizontal direction during a time $t^{1+2\beta_i}$ (i.e., we change the drift from $1/t$ to $1/t^{1+\beta_i}$) and moreover force the vertical coordinate to stay between the boundaries of the tube; standard estimates give that this incurs a cost of order of $1/t$, i.e.\ the probability of this event is of order $e^{-C/t}$ for some constant $C$. Summing over all scales gives a cost of order $N/t$. Note that the union of the tubes of all scales forms a discretized cone with angle approximately $2t$.

Now fix a certain scale $i\in\{1,\ldots,N\}$ and set $x_i = t^{\beta_{i-1}} + t^{1+\beta_i}$ if $i \geq 2$ and $x_i = 0$ if $i=1$. We discretize the tube at scale $i$  into $n_i$ small squares $S^i_0,\ldots,S^i_{n_i}$ of side length $2t^{1+\beta_i}$ (thus, $n_i \approx 1/2t$ and $x_i$ is the center of $S^i_0$) and consider the values of $M_\gamma(S^i_k)$, $k=0,\ldots,n_i$. As before, given $\eta>0$, our aim is to define random variables $W^i_0,\ldots,W^i_{n_i}$ which now verify for every small enough $t$, $\Q$-almost surely, for every $i=1,\ldots,N$,
\begin{enumerate}
\item[(A1')] $t^{1+\beta_i} \sum_{k=1}^{n_i} W^i_k \leq \frac 1 t$,
\item[(A2')] $\frac 1 {t^{1+\beta_i}} \sum_{k=0}^{n_i} (\frac 1 {t^{1+\beta_i}} + W^i_k)^{-1} M_\gamma(S^i_k) \leq  t^{\nu(\gamma)-\eta}$
\item[(A4')] $\sup_{y\in B(x_i,t^{1+\beta_i})}\int_{B(y,t^{1+\beta_i})} \log_+\frac{t^{1+\beta_i}}{|x-y|}\,M_\gamma(dx) \leq t^{\nu(\gamma)-\eta},$
\end{enumerate}
as well as (A5) (assumption (A3) is not needed anymore: it was only used to split the Brownian bridge into two drifted Brownian motions). Here, the exponent $\nu(\gamma)$ will be determined later; we will see that we are allowed to choose $\nu(\gamma) = 1+\gamma^2/4$ for $\gamma^2 \leq 8/3$ but not for larger values of $\gamma$.

We start by verifying (A1') and (A2'). Similarly as before, we define the random variables
\[
W^i_k = t^{-(2+\beta_i + \frac 1 2 (\nu(\gamma)-1)) +\eta/2}\sqrt{M_\gamma(S^i_k)},\quad k=0,\ldots,n_i.
\]
Properties (A1') and (A2') are then verified under the following assumption:
\begin{itemize}
      \item[(A0')] $t^{\frac{1}{2} (1-\nu(\gamma)+\eta)} \sum_{k=0}^{n_i}\sqrt{M_\gamma(S^i_k)} \leq 1$.
\end{itemize}

In order to verify (A0'), we will need a preliminary lemma. Define the functions
\begin{align*}
\psi(x) &= 2\sqrt x - x,\quad x\geq 0.\\
h(a) &= \max_{0\leq x\leq a} \psi(x) = \psi(a\wedge 1) \in [0,1],\quad a\geq 0.\\
h_\gamma(a) &= \frac{\gamma^2}{8} h\left(\frac{8}{\gamma^2}a\right).
\end{align*}
\begin{lemma}
\label{lem:sum}
Let  $a\in(0,1]$. There exists $\delta > 0$, such that for every $i\in\{1,\ldots,N\}$, we have
 \[
  \Pb^X\left(\sum_{k=1}^{\lfloor t^{-a}\rfloor}\sqrt{M_\gamma(S^i_k)} > t^{- a + (1+\beta_i)\left(1+\frac{\gamma^2}4 - h_\gamma\left(\frac{a}{1+\beta_i}\right)\right) -\eta/8}\right) < C t^\delta.
 \]
\end{lemma}
\begin{proof}
 This is fairly standard multifractal analysis which we include for completeness. Throughout the proof, we fix  $i\in\{1,\ldots,N\}$ and set $t_i = t^{1+\beta_i}$. We also assume that $t< 1$. 
By \eqref{eq:chernoff}, we have for every $q\geq0$ and $k\geq1$,
\begin{equation}
 \label{eq:estimate_Mgamma}
 \Pb^X\left(M_\gamma(S_k^i) \geq t_i^{2+\frac{\gamma^2}{2}-q} \right) \leq C t_i^{\frac{1}{2\gamma^2}q^2}.
\end{equation}
Now let $L\in\N$ and $0 = q_1 < q_2 < \cdots < q_L$, later we will fix specific values. For $\ell =0,\ldots,L$, set 
\[
 \Sigma^i_\ell = \sum_{k=1}^{\lfloor t^{-a}\rfloor}\sqrt{M_\gamma(S^i_k)} \cdot \Ind{M_\gamma(S^i_k) \in [t_i^{2+\frac{\gamma^2}{2}-q_\ell},t_i^{2+\frac{\gamma^2}{2}-q_{\ell+1}}]},
\]
where we formally set $q_0 = -\infty$ and $q_{L+1} = +\infty$. 
Note that by definition,
\[
  \Sigma^i_0 \leq t^{-a} t_i^{1+\frac{\gamma^2}{4}}.
\]
Furthermore, for every $\ell\in\{1,\ldots,L\}$ and $\kappa\geq 0$,
\begin{align*}
 \Pb^X\left(\Sigma^i_\ell  \geq t^{-a} t_i^{1+\frac{\gamma^2}{4}-\kappa} \right) 
 &\leq 
 \Pb^X\left(\sum_{k=1}^{\lfloor t^{-a}\rfloor}\Ind{M_\gamma(S^i_k) \geq t_i^{2+\frac{\gamma^2}{2}-q_\ell}}
 \geq t^{-a} t_i^{1+\frac{\gamma^2}{4}-\kappa - (1+\frac{\gamma^2}{4}-\frac 1 2 q_{\ell+1})} \right)\\
 &\leq \left(t^{-a} t_i^{-\kappa +\frac 1 2 q_{\ell+1}} \vee 1\right)^{-1} \times t^{-a} \Pb^X\left(M_\gamma(S^i_1) \geq t_i^{2+\frac{\gamma^2}{2}-q_\ell}\right)\\
 &\leq Ct_i^{\max\left\{\kappa -\frac 1 2 q_{\ell+1}, -(1+\beta_i)^{-1}a\right\} +\frac{1}{2\gamma^2}q_\ell^2},
\end{align*}
where the last inequality follows from \eqref{eq:estimate_Mgamma}. We can and will now require that $q_{\ell+1} - q_\ell < \eta/8$ for every $\ell\in\{1,\ldots,L-1\}$ and that $q_L^2 \geq  2\gamma^2 (1+(1+\beta_i)^{-1} a)$. In this case, for every  $\kappa \geq 0$,
\[
\Pb^X\left(\Sigma^i_L  \geq t^{-a} t_i^{1+\frac{\gamma^2}{4}-\eta/8 - \kappa} \right) \leq Ct_i,
\]
and moreover for every $\ell\in\{1,\ldots,L-1\}$,
\[
 \Pb^X\left(\Sigma^i_\ell  \geq t^{-a} t_i^{1+\frac{\gamma^2}{4}-\eta/8 - \kappa} \right) \leq Ct_i^{\min_{q\geq 0}\left(\max\left\{\kappa + \eta/16 -\frac 1 2 q, -(1+\beta_i)^{-1}a\right\} +\frac{1}{2\gamma^2}q^2\right)}.
\]
In order to finish the proof, it is then enough to show that with $\kappa = h_\gamma\left(\frac{a}{1+\beta_i}\right)$, we have
\[
 \min_{q\geq 0}\left(\max\left\{\kappa + \eta/16 -\frac 1 2 q, -\frac{a}{1+\beta_i}\right\} +\frac{1}{2\gamma^2}q^2\right) > 0,
\]
the lemma then follows from the previous estimates and a union bound. For this, note that with $\widetilde\kappa = \kappa + \eta/16$ and $\widetilde a = \frac{a}{1+\beta_i}$, we have
\begin{align*}
 && \min_{q\geq 0}\left(\max\left\{\widetilde \kappa -\frac 1 2 q, -\widetilde a\right\} +\frac{1}{2\gamma^2}q^2\right) &> 0\\
 \stackrel{q = \frac{\gamma^2}{2}\sqrt{r}}{\Longleftrightarrow} && \min_{r\ge0} \left(\max\left\{\frac{8}{\gamma^2}\widetilde \kappa - 2 \sqrt{r}, -\frac{8}{\gamma^2}\widetilde a\right\} +r\right) &> 0\\
 \Longleftrightarrow && \min_{0\leq r \leq \frac{8}{\gamma^2}\widetilde a} \left(\frac{8}{\gamma^2}\widetilde \kappa - \psi(r)\right) &> 0\\
 \Longleftrightarrow && \widetilde \kappa &> h_\gamma(\widetilde a).
\end{align*}
This finishes the proof.
\end{proof}

Lemma~\ref{lem:sum} now readily gives a sufficient criterion for (A0') to hold for all $i\geq 2$ (the case $i=1$ will be treated separately). Note that for $i\geq 2$, we have $d_{\T}(y,x) > t^{\beta_{i-1}}$ for every $k\in\{0,\ldots,n_i\}$ and every $y\in S^i_k$. By  \eqref{eq:Q} and Lemma~\ref{lem:sum} we then have for some $\delta > 0$,
\[
  \Q\left(\sum_{k=1}^{n_i}\sqrt{M_\gamma(S^i_k)} \geq t^{ - 1 - \frac{\gamma^2}{2}\beta_{i-1} + (1+\beta_i)\left(1+\frac{\gamma^2}4 - h_\gamma((1+\beta_i)^{-1})\right)-\eta/8}\right) \leq Ct^\delta.
\]
In particular, by the Borel--Cantelli lemma, the event on the left-hand side is realized for all small enough dyadic $t$. Standard comparison arguments then imply the existence of a random variable $T_0>0$, such that for $t\leq T_0$,
\[
  \sum_{k=1}^{n_i}\sqrt{M_\gamma(S^i_k)} \leq Ct^{ - 1 - \frac{\gamma^2}{2}\beta_{i-1} + (1+\beta_i)\left(1+\frac{\gamma^2}4 - h_\gamma((1+\beta_i)^{-1})\right)-\eta/8}, \quad\text{$\Q$-a.s.}
\]
We now choose the sequence $(\beta_i)$ such that $\beta_{i-1} \leq \beta_i + (2/\gamma^2)(\eta/4)$ for all $i\geq 2$. The previous equation then gives for $t\leq T_0$,
\[
t^{\frac{1}{2} (1-\nu(\gamma)+\eta)}\sum_{k=1}^{n_i}\sqrt{M_\gamma(S^i_k)} \leq t^{\frac{1}{2} (-1+\gamma^2-\nu(\gamma))+(1+\beta_i)\left(1-\frac{\gamma^2}4 - h_\gamma((1+\beta_i)^{-1})\right)}.
\]
The previous calculations now show for each $i\geq 2$ that the following inequality implies (A0') for $t\leq T_0$:
\begin{itemize}
      \item[(A0'')] $\frac{1}{2} (-1+\gamma^2-\nu(\gamma))+(1+\beta_i)\left(1-\frac{\gamma^2}4 - h_\gamma((1+\beta_i)^{-1})\right)  \geq 0.$
\end{itemize}
Splitting into the cases where $(1+\beta_i)^{-1}$ is less or greater than $\gamma^2/8$, we get the following inequalities equivalent to (A0''):
\begin{align*}
     &\text{(A0a)} & &\frac 1 2 (1+\gamma^2/4-\nu(\gamma)) + (1-\frac 3 8 \gamma^2)\beta_i \geq 0, & \quad(1+\beta_i)^{-1} &\geq \gamma^2/8,\\
    &\text{(A0b)} & &\frac 1 2 (1+\gamma^2-\nu(\gamma))+(1+\beta_i)\left(1-\frac{\gamma^2}4\right)-\frac \gamma {\sqrt 2}\sqrt{1+\beta_i} \geq 0, & \quad(1+\beta_i)^{-1} &\leq \gamma^2/8,
\end{align*}
and furthermore, (A0a) implies (A0b) (for every $\gamma$ and $\beta_i$).

The case $i=1$ needs to be considered separately. In this case, we have the following lemma
\begin{lemma}
\label{lem:i_equals_1}
There exists $T_0>0$, such that for $t\leq T_0$, $\Q$-almost surely,
\[
\sum_{k=1}^{n_1}\sqrt{M_\gamma(S^1_k)} \le
\begin{cases}
t^{\frac{\gamma^2}{2}-1+(1+\beta_1)\left(1-\frac{\gamma^2}4 - h_\gamma((1+\beta_1)^{-1})\right)-\eta/2} & \text{if } \gamma^2 \leq 2 \text{ or } 1+\beta_1 \geq 2\gamma^2\\
t^{(1+\beta_1)\frac{3-\gamma^2}4-\eta/2} & \text{otherwise}.
\end{cases}
\]
\end{lemma}
\begin{proof}
Let $0=\eps_0<\eps_1<\cdots<\eps_L$ such that $\eps_\ell - \eps_{\ell-1} \leq \eta/8$ for $\ell = 1,\ldots,L$ and such that $n_1 = \lfloor t^{-\eps_L}\rfloor$. We further assume that $t$ is small enough such that $|\eps_L - 1| \leq \eta/8$. We then have for any increasing function $f$,
\begin{align*}
\E^\Q\left[f\left(\sum_{k=1}^{n_1} \sqrt{M_\gamma(S^1_k)}\right)\right] &\leq C\E^X\left[f\left(C\sum_{k=1}^{n_1} (kt^{1+\beta_1})^{-\gamma^2/2}\sqrt{M_\gamma(S^1_k)}\right)\right]\\
&\leq C\E^X\left[f\left(C \sum_{\ell=1}^L t^{-\frac{\gamma^2}{2}(1+\beta_1-\eps_{\ell-1})}\sum_{k=\lfloor t^{-\eps_{\ell-1}}\rfloor}^{\lfloor t^{-\eps_\ell} \rfloor}\sqrt{M_\gamma(S^1_k)}\right)\right].
\end{align*}
By Lemma~\ref{lem:sum}, there exists $\delta > 0$, such that with $\Pb^X$-probability at least $1-Ct^\delta$,  for all $\ell\in\{1,\ldots,L\}$,
\[
\sum_{k=\lfloor t^{-\eps_{\ell-1}}\rfloor}^{\lfloor t^{-\eps_\ell} \rfloor}\sqrt{M_\gamma(S^1_k)} \leq \sum_{k=1}^{\lfloor t^{-\eps_\ell} \rfloor}\sqrt{M_\gamma(S^1_k)} \leq t^{-\eps_\ell + (1+\beta_1)\left(1+\frac{\gamma^2}4 - h_\gamma\left(\frac{\eps_\ell}{1+\beta_1}\right)\right) -\eta/8}.
\]
Hence, with $\Pb^X$-probability at least $1-Ct^\delta$
\begin{align*}
 \sum_{\ell=1}^L t^{-\frac{\gamma^2}{2}(1+\beta_1-\eps_{\ell-1})}\sum_{k=\lfloor t^{-\eps_{\ell-1}}\rfloor}^{\lfloor t^{-\eps_\ell} \rfloor}\sqrt{M_\gamma(S^1_k)} &\leq t^{(1+\beta_1)\left(1-\frac{\gamma^2}{4}\right) - \eta/2} \sum_{\ell=0}^{L-1} t^{-\eps_{\ell}\left(1-\frac{\gamma^2}{2}\right) - (1+\beta_1) h_\gamma\left(\frac{\eps_\ell}{1+\beta_1}\right)}\\
&\leq t^{(1+\beta_1)\left(1-\frac{\gamma^2}{4}\right) - \eta/2} L t^{-m(\gamma,\beta_1)},
\end{align*}
where
\[
m(\gamma,\beta_1) = \max_{z\in[0,1]} \overline{m}_{\gamma,\beta_1}(z),\quad \overline{m}_{\gamma,\beta_1}(z)=\left\{\left(1-\frac{\gamma^2}{2}\right)z + (1+\beta_1) h_\gamma\left(\frac{z}{1+\beta_1}\right)\right\}.
\]
By the same arguments as above (considering only dyadic $t$ and using comparison arguments for the other values), the previous inequalities now entail the existence of $T_0>0$, such that
\[
 \sum_{k=1}^{n_1} \sqrt{M_\gamma(S^1_k)} \leq Ct^{(1+\beta_1)\left(1-\frac{\gamma^2}{4}\right) - \eta/2} t^{-m(\gamma,\beta_1)},\quad \text{$\Q$-almost surely.}
\]

It remains to determine $m(\gamma,\beta_1)$. If $\gamma^2 \leq 2$, then the function $\overline{m}_{\gamma,\beta_1}$ is non-decreasing in $z$ and $m(\gamma,\beta_1) = \overline{m}_{\gamma,\beta_1}(1)$. Therefore, suppose that $\gamma^2 > 2$. In this case, the maximum of $\overline{m}_{\gamma,\beta_1}$ is attained at a point $z_0\in(0,(1+\beta_1)\gamma^2/8)$ and in this interval, we have
\[
 \overline{m}_{\gamma,\beta_1}(z) = \left(1-\frac{\gamma^2}{2}\right)z + (1+\beta_1) \frac{\gamma^2}{8}\psi\left(\frac{z}{(1+\beta_1)\frac{\gamma^2}{8}}\right),
\]
where the function $\psi$ was defined above. From this, one easily deduces that $m(\gamma,\beta_1) = \overline{m}_{\gamma,\beta_1}(1)$ if $1+\beta_1 \geq 2\gamma^2$ and $m(\gamma,\beta_1) = (1+\beta_1)/4$ otherwise. This finishes the proof of the lemma.
\end{proof}

The previous lemma directly implies that (A0') is satisfied for $i=1$ if the following inequalities are satisfied:
\[
\begin{cases}
\qquad\qquad\qquad\text{(A0'', $i=1$)} & \text{if } \gamma^2 \leq 2 \text{ or } 1+\beta_1 \geq 2\gamma^2\\
\text{(A0c)}\quad \frac 1 2 (1-\nu(\gamma))+(1+\beta_1)\frac{3-\gamma^2}4 \geq 0 & \text{otherwise}
\end{cases}
\]

As for assumption (A4'), we note that as in the previous section, it is enough to verify the following for $t<T_0$ and $i=1,\ldots,N$:
\begin{itemize}
      \item[(A4'')] $M_\gamma(B(x_i,2t^{1+\beta_i})) \geq t^{\nu(\gamma)-\eta/2}.$
\end{itemize}
Standard results give that for some $T_0>0$, we have for $t<T_0$ and $i=1,\ldots,N$ (remember the convention $\beta_0 = 1+\beta_1$):
\[
M_\gamma(B(x_i,2t^{1+\beta_i})) 
\leq t^{(1+\beta_i)(2+\frac{\gamma^2}{2}) - \beta_{i-1}\gamma^2-\eta/4} 
\leq 
\begin{cases} 
t^{(1+\beta_1)(2-\frac{\gamma^2}{2})-\eta/4} & i=1\\
t^{(1+\beta_i)(2-\frac{\gamma^2}{2}) + \gamma^2 -\eta} & i\geq 2
\end{cases}
\]
since $\beta_{i-1} - \beta_i\leq \eta/(2\gamma^2)$ by assumption for $i\geq 2$. Since $2-\gamma^2/2 > 0$ for all $\gamma < 2$, the previous inequality entails that (A4'') and hence (A4') is satisfied as soon as the following inequalities are satisfied:
\begin{itemize}
   \item[(A4a)] $1+\beta_1 \geq \frac{\nu(\gamma)}{2-\frac{\gamma^2}{2}},$
   \item[(A4b)] $\nu(\gamma) \leq 2+\frac{\gamma^2}{2}.$
\end{itemize}

We now give choices of $\nu(\gamma)$ and $\beta_1$, such that all the previous inequalities are satisfied. For this, we consider three different cases:

\paragraph{Case $\gamma^2 \in [0,8/3]$.} In this case, set $\nu(\gamma) = 1+\gamma^2/4$. Then (A0a) is satisfied for all $i\geq 1$ irrespectively of the choice of $\beta_1$, whence (A0b) as well. Choosing $\beta_1$ large enough, such that either $1+\beta_1 \geq 2\gamma^2$ or (A0c) is verified, it follows that (A0') is verified for all $i\geq 1$. Furthermore, (A4a) can be verified as well by choosing $\beta_1$ large enough. (A4b) is trivially satisfied.

\paragraph{Case $\gamma^2 \in (8/3,3]$.} Recall that we have $h_\gamma(x) = \gamma^2/8$ for $x \leq \gamma^2/8$ and $h_\gamma(x)\downarrow 0$ as $x\to\infty$. If $\gamma^2 > 8/3$, this implies that the minimum of the LHS of (A0'') as a function of $\beta_i$ is satisfied for some $\beta_i$ with $(1+\beta_i)^{-1}\leq \gamma^2/8$. From (A0b), this is easily seen to be the case for 
\[
1+\beta_i = \frac{\gamma^2}{8}\left(1-\frac{\gamma^2}{4}\right)^{-2}.
\]
Plugging this value into the LHS (A0b), the inequality is satisfied if and only if
\begin{equation}
\label{eq:nu_83_3}
\nu(\gamma) \leq 1+\gamma^2 - \frac{\gamma^2}{4}\left(1-\frac{\gamma^2}{4}\right)^{-1}.
\end{equation}
To sum up, the previous arguments give that \eqref{eq:nu_83_3} implies (A0'') for every $i\geq 1$, irrespective of the choice of $\beta_1$. Note that \eqref{eq:nu_83_3} entails (A4b). We can then choose $\beta_1$ large enough such that either $1+\beta_1 \geq 2\gamma^2$ or (A0c) is verified and, moreover, such that (A4a) are verified. This shows that all assumptions are satisfied when setting $\nu(\gamma)$ to be the RHS of \eqref{eq:nu_83_3}.

\paragraph{Case $\gamma^2 \in (3,4)$.} Here we try to choose $\beta_1$ as small as possible. It turns out that the best choices are $\beta_1 = 1$ and $\nu(\gamma) = 4-\gamma^2$. With this choice, equality holds in (A4a) and (A0c). (A4b) is trivially satisfied. Furthermore, one checks by calculation that (A0a) holds for $i=1$ and therefore for all $i\geq 2$ by monotonicity. Since $(1+\beta_1)^{-1} = 1/2 \geq \gamma^2/8$, it follows that (A0'') holds for all $i\geq 2$. All assumptions are therefore verified.

\paragraph{Conclusion.} Assumptions (A1'), (A2') and (A4') hold with $\nu(\gamma)$ as in \eqref{eq:nu_gamma}.

 \appendix
 
\section{An auxiliary result}
  \begin{lemma}\label{laplace}
 Let $\mu, \nu,c >0$.  Then, for some constants 
$C>0$ and $a>0$ (which  depend on $\nu,\mu,c$) and any $t_0>0$ we have
\begin{equation*}
\liminf_{\lambda\to\infty} 
\frac{\int_0^{t_0} e^{-\lambda t ^\nu- \frac{c}{t^\mu} } dt}{
\lambda^{-a}e^{- (\frac{\mu c}{\nu})^{\frac{\nu}{\nu+\mu}} 
(1+\frac{\nu}{\mu}) \lambda^{\frac{\mu}{\nu+\mu}}}}\geq C.
 \end{equation*}
\end{lemma}

\noindent {\it Proof.} The proof follows Laplace's method. Set $F(t,\lambda)=\lambda t^\nu+ct^{-\mu}$.
Then $t\mapsto F(t,\lambda)$ has, for any fixed $\lambda$, a unique minimum, 
achieved at
$t^*=\lambda^{-1/(\mu+\nu)}(c\mu/\nu)^{1/(\mu+\nu)}$.
Making the change 
of variables $t=t^*(1+s)$, one has that 
$F(t,\lambda)=F(t^*,\lambda)+ G(s,\lambda)$,
where the function $G(s,\lambda)$ satisfies $G(0,\lambda)=0$,
$dG(s,\lambda)/ds|_{s=0}=0$ and 
$$d^2G(s,\lambda)/ds^2\leq 2\alpha
 F(t^*,\lambda),$$
for some $\alpha=\alpha(\mu,\nu,c)$, $\delta=\delta(\nu,\mu,c)\in (0,1)$ and all
$s\in (-\delta,\delta)$ and $\lambda\geq 1$.
Since $t^*\to_{\lambda\to\infty} 0$, there exists a 
$\lambda_0=\lambda_0(\mu,\nu,c,t_0)$ so that
for all $\lambda>\lambda_0$ one has 
$t^*(1+\delta)<t_0$.
Therefore, for $\lambda>\lambda_0$,
$$\int_0^{t_0} e^{-\lambda t ^\nu- \frac{c}{t^\mu} } dt
\geq t^* e^{-F(t^*,\lambda)}\int_{-\delta}^{\delta}
e^{-s^2 \alpha F(t^*,\lambda)} ds\geq 
C \frac{t^* e^{-F(t^*,\lambda)}}
{\sqrt{F(t^*,\lambda)}}\,.$$
The claim follows. \qed

\begin{remark}
It is not hard to modify the proof of Lemma \ref{laplace} in order to show that
in fact, under the conditions of the lemma, 
 $$\int_0^{t_0} e^{-\lambda t ^\nu- \frac{c}{t^\mu} } dt \underset{\lambda \to \infty }{\sim} C\lambda^{-a}e^{- (\frac{\mu c}{\nu})^{\frac{\nu}{\nu+\mu}} (1+\frac{\nu}{\mu}) \lambda^{\frac{\mu}{\nu+\mu}}}.
$$
Since we do not use this stronger estimate, we do not provide the details here.
\end{remark}

\end{document}